\documentclass[10pt, reqno]{amsart}
\usepackage{graphicx, amssymb, amsmath, amsthm, color, slashed, cite}
\numberwithin{equation}{section}

\usepackage{hyperref}

\let\Re=\undefined\DeclareMathOperator*{\Re}{Re}
\let\Im=\undefined\DeclareMathOperator*{\Im}{Im}

\newcommand{\R}{\mathbb{R}}
\newcommand{\C}{\mathbb{C}}

\newcommand{\eps}{\varepsilon}

\newtheorem{theorem}{Theorem}[section]

\newtheorem{lemma}[theorem]{Lemma}

\newtheorem{corollary}[theorem]{Corollary}

\newtheorem{proposition}[theorem]{Proposition}

\theoremstyle{definition}
\newtheorem{definition}[theorem]{Definition}
\newtheorem{remark}[theorem]{Remark}

\theoremstyle{remark}

\newcommand{\qtq}[1]{\quad\text{#1}\quad}

\begin{document}

\title[Inhomogeneous NLS]{Dynamics of the non-radial energy-critical inhomogeneous NLS}

\author[C. M. Guzm\'an]{Carlos M. Guzm\'an}
\address{Department of Mathematics, UFF, Brazil}
\email{carlos.guz.j@gmail.com}
\author[C. Xu]{Chenbgin Xu}
\address{School of Mathematics and Statistics, Qinghai Normal University, Xining, Qinghai 810008, P.R.
China}
\email{xcbsph@163.com}

\begin{abstract} We consider the focusing inhomogeneous nonlinear Schr\"odinger equation
\[
i\partial_t u + \Delta u + |x|^{-b}|u|^\alpha u = 0\qtq{on}\R\times\R^N,
\]
with $\alpha=\tfrac{4-2b}{N-2}$, $N=\{3,4,5\}$  and
$0<b\leq \min\Big\{\tfrac{6-N}{2},\tfrac{4}{N}$\Big\}. This paper establishes global well-posedness and scattering for the non-radial energy-critical case in $\dot{H}^1(\R^N)$. It extends the previous research by Murphy and the first author \cite{GM}, which focused on the case $(N,\alpha,b)=(3,2,1)$. The novelty here, beyond considering higher dimensions, lies in our assumption of the condition $\sup_{t\in I}\|\nabla u(t)\|_{L^2}<\|\nabla Q\|_{L^2}$, which is weaker than the condition stated in \cite{Guzman}. Consequently, if a solution has energy and kinetic energy less than the ground state $Q$ at some point, then the solution is global and scatters. Moreover, we show scattering for the defocusing case. On the other hand, in this work, we also investigate the blow-up issue with nonradial data for $N\geq 3$ in $H^1(\mathbb{R}^N)$. This implies that our result holds without classical assumptions such as spherically symmetric data or $|x|u_0 \in L^2(\mathbb{R}^N)$.

\

\noindent Mathematics Subject Classification. 35A01, 35QA55, 35P25.\quad\quad\quad

\end{abstract}

\keywords{Key words. Inhomogeneous nonlinear Schr\"odinger equation; Global well-posedness; Scattering.}

\maketitle

\section{Introduction}

In this paper, we study the energy-critical inhomogeneous nonlinear Schr\"odinger equation (INLS) for non-radial initial data
\begin{equation}\label{INLS}
\begin{cases}
& i\partial_t u + \Delta u + |x|^{-b}|u|^\alpha u = 0, \\
& u|_{t=0}=u_0 \in H^1(\R^N),
\end{cases}
\end{equation}
where $\alpha=\frac{4-2b}{N-2}$ and $b>0$.

The equation \eqref{INLS} enjoys the scaling symmetry $u(t,x)\mapsto \lambda^{\frac{2-b}{\alpha}}u(\lambda^2 t,\lambda x)$, which identifies the unique invariant homogeneous $L^2$-based Sobolev space of initial data as $\dot H^{s_c}$, where
\[
s_c:=\tfrac{N}{2}-\tfrac{2-b}{\alpha}.
\]
When $s_c=1$, the critical space is $\dot H^1(\R^N)$, which is naturally associated to the conserved \emph{energy} of solutions, defined by\footnote{We denote $\dot H^1(\R^N)$ simply as $\dot H^1$.}
\[
E[u] = \int \tfrac12 |\nabla u|^2 - \tfrac1{\alpha+2}|x|^{-b}|u|^{\alpha+2}\,dx.
\]

Our first and main interest in this work is the scattering problem for \eqref{INLS} with general data. Here we say that a solution to \eqref{INLS} \emph{scatters} (in $\dot{H}^1$) if there exist $u_\pm\in \dot{H}^1$ such that
\[
\lim_{t\to\pm\infty}\|u(t) - e^{it\Delta}u_\pm\|_{\dot{H}^1}=0,
\]
where $e^{it\Delta}$ is the free Schr\"odinger propagator.

Note that when $b = 0$, the equation above reduces to the classical nonlinear Schrödinger equation (NLS), which appears in the description of nonlinear waves for various physical phenomena. It has drawn considerable attention in the mathematical community over the last three decades. The asymptotic behavior of the solution (scattering) for the energy-critical case was obtained by Kenig and Merle \cite{KM}, who pioneered a strategy now known as the `roteiro Kenig-Merle'. In their work, they showed global well-posedness and scattering below the ground state for the quintic NLS in three dimensions, assuming radial data. Killip and Visan \cite{KV} further extended this result for $N \geq 5$, considering the non-radial setting. While Dodson \cite{DodsonE} studied the $N = 4$ case, the three-dimensional scenario remains an open question.

On the other hand, the inhomogeneous NLS model has also been the subject of a great deal of recent mathematical interest. The well-posedness problem has been studied in works such as \cite{CFG20, Farah, Guzman, GENSTU}. Additionally, the dynamics of the solution have also been explored. In \cite{Campos, FG, FG2} proved scattering for \eqref{INLS} in the $L^2$-supercritical and $\dot{H}^1$ subcritical regime (inter-critical case), i.e., $\frac{4-2b}{N}<\alpha<\frac{4-2b}{N-2}$ assuming radial initial data. In \cite{MMZ} and \cite{CFGM} removed that condition, that is, they proved scattering for general data. Very recently,  Murphy and the first author \cite{GM} treated the energy-critical case. They studied the scattering theory for \eqref{INLS} in three space dimensions below the ground state. More precisely,
\begin{theorem}\label{PT} Let \eqref{INLS} with  $(N,\alpha,b)=(3,2,1)$. Suppose $u_0\in \dot{H}^1$ obeys
\begin{equation}\label{threshold}
E[u_0]<E[W]\qquad  \textnormal{and} \qquad \|\nabla u_0\|_{L^2}<\|\nabla W\|_{L^2},
\end{equation}
where $W$ denotes the ground-state solution to
\[
\Delta W + |x|^{-b}W^{\alpha+1}=0.
\]
Then the corresponding solution $u$ to \eqref{INLS} is global-in-time and scatters in $\dot{H}^1$.
\end{theorem}
This result represents an expansion of those obtained in \cite{ChoHongLee,ChoLee} from the radial to the non-radial context. It's worth noting that Theorem \ref{PT} establishes scattering in $3D$ considering non-radial data; however, for the homogeneous model, it is still open. In our paper, we further extend the result of \cite{GM} to higher dimensions, $N=\{3, 4, 5\}$, and widen the parameter $b$ range in dimension $N=3$. Additionally, we introduce a novelty in the condition, assuming a weaker condition \eqref{kinetic energy} (first introduced by Killip-Visan \cite{KV}) than the one stated in \cite{GM}. The result reads as follows.

\begin{theorem}\label{T} Let $N=\{3,4,5\}$, $0<b\leq\min\{\tfrac{6-N}{2},\tfrac{4}{N}\}$ and $\alpha=\tfrac{4-2b}{N-2}$. If $u$ is a solution of \eqref{INLS} with initial data $u_0$ and maximal lifespan $I$ such that
\begin{equation}\label{kinetic energy}
\sup_{t\in I}\|\nabla u(t)\|_{L^2}<\|\nabla W\|_{L^2},
\end{equation}
then u is a global solution and $\|u\|_{L_{t,x}^{ \tfrac{2(N+2)}{N-2}} }$ is bounded. In addition, the solution scatters in $\dot{H}^1$.
\end{theorem}

This result extends two previous works. Specifically, when $b=0$, the result derived corresponds to that of Killip-Visan \cite{KV}. Furthermore, It also extends the findings obtained in \cite{GM} for $N\geq 4$. In the case where $N=3$, we extend the range from $b=1$ to $0<b<\frac{4}{3}$. The proof of Theorem \ref{T} follows the general outline in \cite{GM} (Kenig-Merle roadmap). As we will explain, we need two new ingredients to address the non-radial case with the kinetic energy condition \eqref{kinetic energy}, along with results on the Cauchy problem, such as global small data, and stability theory, among others. The first establishes the existence of scattering solutions to \eqref{INLS} associated with initial data located sufficiently far from the origin (see Proposition \ref{P:embed} below). The second and main ingredient is the Palais-Smale condition (Proposition \ref{PS}), which establishes blow-up solutions for a sequence of solutions of \eqref{INLS} with a compactness property. To this end, we need the technical results such as a bad profile and Kinetic energy decoupling. Both ingredients enable us to construct a minimal blow-up solution (or critical solution). In doing so, we will encounter several challenges with minimal kinetic energy, different from those developed in \cite{GM}. These difficulties are related to the fact that, unlike energy, kinetic energy is not conserved.


Next, we sketch the proof of Theorem \ref{T}. Let \( K > 0 \), we define
\[
L(K) = \sup_u \left\{ \|u(t)\|_{S(I)} : \sup_{t \in I} \|u(t)\|^2_{\dot{H}^1} \leq K \right\}.
\]
The small data theory of \eqref{INLS} (Proposition \ref{GWPCH1}-(ii)), implies there exists $\delta > 0$ such that: if $\|u_0\|_{\dot{H}^1}< \delta$, then u exists globally and $\|u\|_{S(\R)} \leq 2\delta$. Note that \( L(K) \) is a continuous function concerning \( K \) from the local theory and stability result, thus there must exist a unique critical kinetic energy $E_c$ such that 
$$ 
L(K) < \infty \;\; \textnormal{if}\;\;  K < K_c\;\;\textnormal{and}\;\;L(K) = \infty \;\; \textnormal{if}\;\;  K \geq K_c,
$$
where
\[
K_c = \inf \{ K : L(K) = \infty \}.
\]
In particular, if \( \sup_{t \in I} \|u(t)\|^2_{\dot{H}^1} < K_c \), then the maximal lifespan \( I = \R \) and \( \|u\|_{S(\R)} < \infty \). Hence, Theorem 1.2 is equivalent to
\[
K_c = \|W\|^2_{\dot{H}^1}.
\]

The scattering result is achieved by way of contradiction. If $K_c < \|W\|^2_{\dot{H}^1}$, then there exists a sequence of blow-up solutions $u_n$. Thus, we decompose their initial data $u_n(0)$ into linear profiles at the $\dot{H}^1$ level (using \eqref{P:LPD}), that is,
$$u_n(0)=\sum_{j=1}^{J}g_n^j[e^{it_n^j\Delta}\phi^j]+w_n^J.
$$
Using the first ingredient (Proposition \ref{P:embed}), one can build a corresponding nonlinear profile $v_n^j$ from each linear profile. By the stability result, the blow-up solutions guarantee the existence of a bad profile (see Lemma \ref{bp}), that is, a nonlinear profile that blow-up in scattering norm (see our second and main ingredient, Proposition \ref{PS}). Next, one can construct the critical solution $u_c$, see Proposition \ref{MS}. This solution exhibits a precompactness property in $\dot{H}^1$ up to scaling symmetry. However, Virial-type arguments and conservation laws preclude the existence of $u_c$ (as detailed in Claims 1 and 2 at the end of the text). This contradiction confirms the scattering result of \eqref{INLS} and establishes the theorem.

As a consequence, we have, if a solution has both energy and kinetic energy less than those of the ground state W at some point in time (for example, at $t=0$), then the solution is global and scatters in $\dot{H}^1$.   

\begin{corollary}\label{corollary} Let $N=\{3,4,5\}$, $0<b\leq\min\{\tfrac{6-N}{2},\tfrac{4}{N}\}$ and $\alpha=\tfrac{4-2b}{N-2}$. Suppose $u_0\in \dot{H}^1$ satisfies \eqref{threshold}, then the corresponding solution $u$ to \eqref{INLS} is globally defined for all time and scatters in $\dot{H}^1$.
\end{corollary}

The idea behind the proof is that \eqref{threshold} implies \eqref{kinetic energy} (see Remark \ref{proof od corollary} for details). This corollary extends the Kenig-Merle result \cite{KM} from homogeneous to inhomogeneous settings, assuming radial symmetry, and generalizes the result obtained in \cite{GM} to higher dimensions.

We would like to note that the proof of Theorem \ref{T} adapts without difficulty
(with some simplifications) to the defocusing case. Here we remove the conditions given in \eqref{threshold} because the norm $\|\cdot\|_{\dot{H}^1}$ is always bounded and all the energies are positive. Thus, we have the following result.
\begin{corollary}\label{2} Let $N=\{3,4,5\}$ and $0<b\leq\min\{\tfrac{6-N}{2},\tfrac{4}{N}\}$. 
For any $u_0\in \dot{H}^1$, there exists a unique global
solution to
$$
i\partial_t u + \Delta u - |x|^{-b}|u|^\alpha u = 0,
$$
with initial data $u_0$. Moreover, this global solution scatters in $\dot{H}^1$.
\end{corollary}

This corollary extends Park's result \cite{Park} from the radial to the non-radial setting. It's worth noting that this method differs from the one used in the previous work. It builds upon the ideas developed by Tao \cite{Tao} for the homogeneous equation.

\

Finally, we study the blow-up problem in the energy space $H^1$, also considering non-radial initial data. Before stating the result, it's crucial to address local well-posedness. Since we focus on blow-up in the inhomogeneous Sobolev space $H^1$, we establish the local result for all dimensions, $N\geq 3$. This is possible because the metric employed in this analysis is defined by $\|u-v\|_{W(I)}$, rather than $\|\nabla(u-v)\|_{W(I)}+\|u-v\|_{S(I)}$ which is used in $\dot{H}^1$. See Lemma \ref{LLWPH1}. We state the result 

\begin{theorem}\label{Blow-up}
Let $N\geq 3$, $0<b\leq \frac{4}{N}$ and $\alpha=\tfrac{4-2b}{N-2}$. Suppose $u_0\in H^1(\R^N)$ obeys
\begin{equation*}
E[u_0]<E[W]\qquad  \textnormal{and} \qquad \|\nabla u_0\|_{L^2}\geq\|\nabla W\|_{L^2}.
\end{equation*}
Let $u(t)$ denote the solution to \eqref{INLS} and $I_{\max}$ denote the maximal time interval of existence. Then the following statements hold:
\\

\textbf{(i)} If $b\geq \frac{4}{N}$, then $I_{\max}$ is finite;
\\

\textbf{(ii)} If $0<b<\frac{4}{N}$, then $I_{\max}$ is finite, or $I_{\max}=\R$, for any $T>\infty$, there exists $C>0$ such that
$$\sup_{t\in[0,T]}\|\nabla u(t)\|_{L_x^2}\geq CT^{\frac{2}{N\alpha-4}}.$$
\end{theorem}

Note that the result holds for all dimensions $N\geq 3$. The proof is based on a contradiction argument using localized virial estimates. Here, we do not require assumptions such as spherically symmetric data or $|x|u_0\in L^2$. We take advantage of the decay of the nonlinear term outside a large ball.

\ 

The rest of this paper is organized as follows: In Section~\ref{S:notation} we first introduce some basic notation. We also discuss the well-posedness theory and stability theory for \eqref{INLS} and some variational analysis related to the ground state and the virial identity. In Section~\ref{S:exist}, we will prove that there exists a minimal energy blow-up solution to \eqref{INLS} under the assumption that Theorem \ref{T} fails and prove the main theorem. Here we introduce the new ingredients for our proof. Finally, in Section ~\ref{blowup} we show our second main result (Theorem \ref{Blow-up}).

\

\subsection*{Acknowledgements} C.M.G. was partially supported by Conselho Nacional de Desenvolvimento Científico e Tecnologico - CNPq and Fundação de Amparo à Pesquisa do Estado do Rio de Janeiro - FAPERJ (Brazil). C. Xu was partially supported by Qinghai Natural Science Foundation (No.2024-ZJ-976).

 \ 
 
\section{Notation and preliminaries}\label{S:notation}

Let's begin this section by introducing some notations that will be used throughout this paper. Subsequently, we establish the local well-posedness and investigate the stability result. Finally, we study the variational analysis associated with the equation.

We write $a\lesssim b$ to denote $a\leq cb$ for some $c>0$, denoting dependence on various parameters with subscripts when necessary. If $a\lesssim b\lesssim a$, we write $a\sim b$. 

We utilize the standard Lebesgue spaces $L^p$, the mixed Lebesgue spaces $L_t^q L_x^r$, as well as the homogeneous and inhomogeneous Sobolev spaces $\dot H^{s,r}$ and $H^{s,r}$.  When $r=2$, we write $\dot H^{s,2}=\dot H^s$ and $H^{s,2}=H^s$.  If necessary, we use subscripts to specify which variable we are concerned with. We use $'$ to denote the H\"older dual.

We also need the standard Littlewood--Paley projections $P_{\leq M}$, defined as Fourier multipliers, with the multiplier corresponding to a smooth cutoff to the
region $\{|\xi|\leq M\}$. In particular, we use the following Bernstein type inequalities
\begin{equation*}\label{BerIn}
\||\nabla|^sP_{\leq M} f\|_{L^p_x}\lesssim M^s\|f\|_{L^p_x}.
\end{equation*}
\begin{equation*}
\|P_{M} f\|_{L^p_x}\lesssim M^{\frac{N}{q}-\frac{N}{p}}\|f\|_{L^q_x}.
\end{equation*}

Finally, we use the following estimates: if
 $F(x,z)=|x|^{-b}|z|^\alpha z$,   then (see details in \cite[Remark 2.6]{Guzman} and \cite[Remark 2.5]{FG})
\begin{equation}\label{FEI}
 |F(x,z)-F(x,w)|\lesssim |x|^{-b}\left( |z|^\alpha+ |w|^\alpha \right)|z-w|
\end{equation}
and
\begin{equation}\label{SECONDEI}
\left|\nabla \left(F(x,z)-F(x,w)\right)\right|\lesssim  |x|^{-b-1}(|z|^{\alpha}+|w|^{\alpha})|z-w|+|x|^{-b}|z|^\alpha|\nabla (z- w)|+E,
\end{equation}
where
\begin{eqnarray*}
 E &\lesssim& \left\{\begin{array}{cl}
 |x|^{-b}\left(|z|^{\alpha-1}+|w|^{\alpha-1}\right)|\nabla w||z-w| & \textnormal{if}\;\;\;\alpha > 1 \vspace{0.2cm} \\
|x|^{-b}|\nabla w||z-w|^{\alpha} & \textnormal{if}\;\;\;0<\alpha\leq 1.
\end{array}\right.
\end{eqnarray*}

\ 

\subsection{Well-posedness theory}

To discuss the well-posedness theory for \eqref{INLS}, we first recall the Strichartz estimates (see e.g. \cite{Cazenave, Kato, Fochi}):
\begin{equation}\label{SE1}
\qquad \qquad \;\;\|e^{it\Delta}f\|_{L^q_IL^r_x}  \lesssim \|f\|_{L^2},
\end{equation}
\begin{equation}\label{SE2}
\biggl\| \int_0^t e^{i(t-t')\Delta} g(t')\,dt'\biggr\|_{L^q_IL^r_x} \lesssim \|g\|_{L^{q_1'}_IL^{r_1'}_x}.
\end{equation}
Here $(q,r)$ and $(q_1,r_1)$ are pairs $S$-admissible (or $L^2$-admissible), that is, $\tfrac{2}{q}+\tfrac{N}{r}=\tfrac{N}{2}$ and
\[
\begin{cases} 2\leq r\leq\tfrac{2N}{N-2}, & N\geq 3, \\
2 \leq r < \infty & N=2, \\
2\leq r<\infty & N=1.
\end{cases}
\]
They are one of the essential ingredients in the well-posedness theory for the Cauchy problem \eqref{INLS}. Before stating the results to \eqref{INLS} we first establish an estimate on the nonlinearity. To do that, we define the numbers 
\begin{equation*}
q_0=\frac{2(N+2)(b+1)}{bN+N-2}\qquad r_0=\frac{2N(N+2)(b+1)}{N^2+bN^2+4}\quad \textnormal{and}\quad \bar{r}=\frac{2(N+2)}{N-2}.
\end{equation*}
Denote $W(I)=L_I^{q_0}L_x^{r_0}$ and define the scattering size of a solution to \eqref{INLS} by
\begin{equation}\label{scatteringsize}
\|u\|_{S(I)}=\|u\|_{L_I^{\frac{2(N+2)}{N-2}}L^{\frac{2(N+2)}{N-2}}_x }.  \end{equation}

\begin{lemma}\label{LC}
Assume $N=\{3,4,5\}$, $\alpha=\frac{4-2b}{N-2}$ and $0<b\leq\frac{4}{N}$. Then the following statement holds.
$$\left\||x|^{-b}|f|^\alpha  g\right\|_{L^2_IL^{\frac{2N}{N+2}}_x}\;\lesssim \;\|\nabla f\|^b_{W(I)} \|f\|^{\alpha-b}_{S(I)}\| g\|_{W(I)}.$$
\begin{proof}
Observe that
\begin{equation*}
\frac{N+2}{2N}=\frac{\alpha-b}{\bar{r}}+\frac{b}{r_0}+\frac{1}{r_0}\qquad \textnormal{and}\qquad \frac{1}{2}=\frac{\alpha-b}{\bar{r}}+\frac{b}{q_0}+\frac{1}{q_0} \;,
\end{equation*}
where $\bar{r}=\frac{2(N+2)}{(N-2)}$. It follows from H\"older and Hardy inequalities that
\begin{eqnarray*}
\left\||x|^{-b}|f|^\alpha  g\right\|_{L^2_IL_x^{\frac{2N}{N+2}}}
&\leq & \||x|^{-1}f\|_{L^{q_0}_IL^{r_0}_x}^{b} \|f\|^{\alpha-b}_{L_I^{\bar{r}}L_x^{\bar{r}}}\| g\|_{L^{q_0}_IL^{r_0}_x}
\\
&\lesssim &\|\nabla f\|^{b}_{W(I)} \|f\|^{\alpha-b}_{S(I)}\| g\|_{W(I)}.
\end{eqnarray*}
\end{proof}
\end{lemma}

Note that the condition $b\leq \frac{4}{N}$ implies $\alpha-b\geq 0$. Now, we state the well-posedness result.

\begin{proposition}\label{GWPCH1}
Assume $N=\{3,4,5\}$, $\alpha=\frac{4-2b}{N-2}$ and $0<b\leq \min\{\frac{4}{N},\frac{6-N}{2}\}$. 
\begin{itemize}
\item [(i)] For any $u_0 \in \dot{H}^1$, there
exists $T = T (u_0)$ and a unique solution $u$ with initial data $u_0$ satisfying
$$
u \in L^q_{loc}\left((-T, T \right), \dot{H}^{1,r}),\;\;\; \forall\; (q, r) \;\;S\textnormal{-admissible}.
$$
\item [(ii)] In particular, there exists $\delta>0$ such that if $$\|e^{it\Delta}u_0\|_{S\left([0,\infty)\right)}+\|e^{it\Delta}u_0\|_{W\left([0,\infty)\right)}<\delta,$$ then the solution $u$ to \eqref{INLS} is forward global and obeys
\begin{equation*}\label{NGWP3}
\|u\|_{S\left([0,\infty)\right)}\leq  2 \delta. 
\end{equation*}
The analogous statement holds backward in time or on all of $\R$.
\end{itemize}
\begin{proof}
To do that we use the contraction mapping argument. We only show (ii). Item (i) is similar. Define
$$
S_{\rho}=\{u\in C(I;\dot{H}^1(\mathbb{R}^N))\;:\; \|u\|_{S^1(I)}\leq \rho \},
$$
where
\begin{equation}\label{normS1}
\|u\|_{S^1(I)}:=\|u\|_{S(I)}+\|\nabla u\|_{W(I)}.
\end{equation}
We will next choose $\delta$, $\rho$ so that $G$ defined by
\begin{equation}\label{OPERATOR} 
G(u)(t)=e^{it\Delta}u_0+i\lambda \int_0^t e^{i(t-t')\Delta}|x|^{-b}|u(t')|^\alpha u(t')dt'
\end{equation}
is a contraction on $S_{\rho}$ equipped with the metric
$$
d(u,v):=\| u-v\|_{S^1(I)}.
$$

\ Hence, applying the Sobolev embedding and the Strichartz estimates \eqref{SE1}-\eqref{SE2}, it follows that
\begin{eqnarray*}
\|G(u)\|_{S(I)}&\leq&  \|e^{it\Delta}u_0\|_{S(I)}+\left\|\nabla \int_0^t e^{i(t-s)\Delta} F(x,u)ds\right\|_{L_I^{\frac{2(N+2)}{N-2}}L_x^{\frac{2N(N+2)}{N^2+4}}}\\
&\leq&  \|e^{it\Delta}u_0\|_{S(I)}+c\left\|\nabla F(x,u)\right\|_{L_I^{2}L_x^{\frac{2N}{N+2}}},
\end{eqnarray*}
where we have used the fact that $\left(\frac{2(N+2)}{N-2},\frac{2N(N+2)}{N^2+4}\right)$ is $S$-admissible. Moreover,
\begin{equation*}
\|\nabla G(u)\|_{W(I)}\leq \|\nabla e^{it\Delta}u_0\|_{W(I)}+ c\|\nabla F(x,u)\|_{L^2_IL_x^{\frac{2N}{N+2}}}.
\end{equation*}
Since
\begin{equation*}
|\nabla F (x,u)| \leq  |x|^{-b}|u|^{\alpha} |\nabla u| + |x|^{-b}|u|^{\alpha} ||x|^{-1} u|
\end{equation*}
and combining Lemma \ref{LC} together with the Hardy inequality, we obtain
\begin{eqnarray*}
\|\nabla F(x,u)\|_{L^2_IL_x^{\frac{2N}{N+2}}} &\leq & c \|u\|^{\alpha-b}_{S(I)}\|\nabla u\|^{b+1}_{W(I)}.
\end{eqnarray*}
Thus,
\begin{equation*}
\begin{split}
\|G(u)\|_{S(I)}
&\leq \|e^{it\Delta}u_0\|_{S(I)} +c \|u\|^{\alpha-b}_{S(I)}\|\nabla u\|^{b+1}_{W(I)}
\end{split}
\end{equation*}
and
\begin{equation}\label{TGHS11}
\begin{split}
\|\nabla G(u)\|_{W(I)}&\leq  \|\nabla e^{it\Delta}u_0\|_{W(I)}+c\|u\|^{\alpha-b}_{S(I)}\|\nabla u\|^{b+1}_{W(I)}.
\end{split}
\end{equation}
By summing the last two inequalities, one has
\begin{equation*}
\|G(u)\|_{S^1(I)}
\leq \|e^{it\Delta}u_0\|_{S^1(I)} +c \|u\|^{\alpha+1}_{S(I)}<\delta+c\rho^{\alpha+1}.
\end{equation*}
Choosing $\delta=\frac{\rho}{4}$ and $c\rho^{\alpha}<\frac{1}{2}$ we obtain $\| G(u)\|_{S^1(I)}\leq \rho$, which means to say   $G(u)\in S_{\rho}$.

\ To complete the proof we show that $G$ is a contraction on $S_{\rho}$. Repeating the above computations we get
\begin{equation}\label{Contrction1}
d(G(u),G(v))
\leq c \|\nabla \left(F(x,u)- F(x,v)\right)\|_{L^2_IL^{\frac{2N}{N+2}}_x}.
\end{equation}
The relation \eqref{SECONDEI} yields (using the fact $\alpha\geq1$)
\begin{eqnarray*}
\left|\nabla \left(F(x,u)-F(x,v)\right)\right|&\lesssim & |x|^{-b}(|u|^{\alpha}+|v|^{\alpha})||x|^{-1}(u-v)|+|x|^{-b}|u|^\alpha|\nabla (u- v)|+E\\
&\lesssim & |x|^{-b}(|u|^{\alpha}+|v|^{\alpha})|\nabla(u-v)|
+E,
\end{eqnarray*}
where $E \lesssim |x|^{-b}\left(|u|^{\alpha-1}+|v|^{\alpha-1}\right)|\nabla v||u-v|$.
Hence, \eqref{Contrction1} and Lemma \ref{LC} lead to
\begin{equation*}
d(G(u),G(v))\;\leq c\;\left( \|\nabla u\|^b_{W(I)} \|u\|^{\alpha-b}_{S(I)}+\|\nabla v\|^b_{W(I)} \|v\|^{\alpha-b}_{S(I)}\right) \|\nabla (u-v)\|_{W(I)}+E_1,
\end{equation*}
where
\begin{align*}
E_1 \lesssim
\begin{cases}
  \left(\|u\|^{\alpha-1}_{S(I)}+\|v\|^{\alpha-1}_{S(I)}\right) \|\nabla v\|_{W(I)}\|\nabla(u-v)\|_{W(I)}^b\|u-v\|_{S(I)}^{1-b},\ \text{if $b<1$};\\
  \left(\|\nabla u\|_{W(I)}^{b-1}\|u\|_{S(I)}^{\alpha-b}+\|\nabla v\|_{W(I)}^{b-1}\|v\|_{S(I)}^{\alpha-b}\right)\|\nabla v\|_{W(I)}\|\nabla (u-v)\|_{W(I)},\ \text{if $b\geq1$}.
\end{cases}
\end{align*}
Therefore, if $u, v \in S_{\rho}$ then
$$
d(G(u),G(v))\;\leq 4c \rho^{\alpha}d(u,v),
$$
which means that  $G$ is a contraction (we choose $4c\rho^\alpha<1$). So, by the contraction mapping principle, $G$ has a unique fixed point $u\in S_{\rho}$.
\end{proof}
\end{proposition}

Note that the conditions $b\leq \frac{6-N}{2}$ and $b\leq \frac{4}{N}$ in Proposition \eqref{GWPCH1} arise due to $\alpha\geq 1$ and $\alpha-b\geq 0$, respectively.

Similarly as before, using Strichartz estimates and Lemma \ref{LC} we also obtain

\begin{corollary}

\begin{itemize}
\item For any $\psi\in \dot{H}^1$, there exist $T > 0$ and a solution $u : (T ,\infty)\times \R^N \rightarrow \C$ to \eqref{INLS} obeying
$e^{-it\Delta}u(t) \rightarrow \psi$ in $\dot{H}^1$ as $t \rightarrow \infty$. The analogous statement holds backward. 
\item Scattering criterion: if the norms $\|u\|_{L^\infty_t\dot{H}^1_x}$ and $\|u\|_{S[0,\infty)}$ are bounded, then the solution $u$ scatters forward in time. 
\end{itemize} 
\end{corollary}



We now address the
stability theory (or long-time perturbation) for the energy-critical inhomogeneous NLS.
\begin{proposition}[Stability]\label{stability} Let $N=\{3,4,5\}$, $0<b\leq \min\{\frac{4}N,\tfrac{6-N}{2}\}$, and $\alpha=\tfrac{4-2b}{N-2}$. Let $0\in I\subseteq \mathbb{R}$ and $\tilde u:I\times \mathbb{R}^N\to\C$ be a solution to
\begin{equation*}
i\partial_t \tilde{u} +\Delta \tilde{u} + |x|^{-b} |\tilde{u}|^\alpha \tilde{u}=e,
\end{equation*}
with initial data $\tilde{u}_0\in \dot{H}^1(\mathbb{R}^N)$ satisfying \begin{equation*}
\sup_{t\in I}  \|\tilde{u}\|_{\dot{H}^1_x}\leq M \;\;\; \textnormal{and}\;\;\; \|\tilde{u}\|_{S(I)}\leq L,\;\;\; \textnormal{for}\;\;M,L>0.
\end{equation*}
Let $u_0\in \dot{H}^1(\mathbb{R}^N)$ such that
\begin{equation*}
\|u_0-\tilde{u}_0\|_{\dot{H}^1}\leq M',\quad \|\nabla e\|_{L^2_IL_x^{\tfrac{2N}{N+2}}}\leq \varepsilon  \qtq{and} \|e^{it\Delta}(u_0-\tilde{u}_0)\|_{S^1(I)}\leq \varepsilon,
\end{equation*}
for\footnote{Recall the norm $\|\cdot\|_{S^1(I)}$ is defined in \eqref{normS1}.} some positive constant $M'$ and some $0<\varepsilon<\varepsilon_1=\varepsilon_1(M,M',L)$.

\indent Then, there exists a unique solution $u:I\times\R^N\to\C$ with $u|_{t=0}=u_0$ obeying	
\begin{align*}
\|u-\tilde{u}\|_{S^1(I)}\lesssim_{M,M',L}\varepsilon\quad \textnormal{and}\quad
\|u\|_{S(I)} \lesssim_{M,M',L} 1.
\end{align*}
\end{proposition}

The stability result for \eqref{INLS} will be obtained iteratively from a short-time perturbation result.
\begin{lemma}\label{STP}{\bf (Short-time perturbation).}
Assume that assumptions in Proposition \ref{stability} hold.  Let $I\subseteq \mathbb{R}$ be a time interval containing zero and let $\widetilde{u}$ be a solution of
\begin{equation*}\label{PE}
i\partial_t \widetilde{u} +\Delta \widetilde{u} + |x|^{-b} |\widetilde{u}|^\alpha \widetilde{u} =e,
\end{equation*}
 defined on $I\times \mathbb{R}^N$, with initial data $\widetilde{u}_0\in \dot{H}^1(\mathbb{R}^N)$, and satisfying
\begin{equation}\label{PC11}  
\sup_{t\in I}  \|\tilde{u}\|_{\dot{H}^1_x}\leq M \;\; \textnormal{and}\;\; \|\tilde{u}\|_{S(I)}\leq \varepsilon,
\end{equation}
for some positive constant $M$ and some small $\varepsilon>0$. Let $u_0\in \dot{H}^1(\mathbb{R}^N)$ be such that
\begin{equation}\label{PC22}
\|u_0-\widetilde{u}_0\|_{\dot{H}^1}\leq M',\;\; \|\nabla e\|_{L^2_IL_x^{\tfrac{2N}{N+2}}}\leq \varepsilon  \;\;{and}\;\; \|e^{it\Delta}(u_0-\tilde{u}_0)\|_{S^1(I)}\leq \varepsilon,\;\;\textnormal{for }\; M'>0.
\end{equation}
\indent There exists $\varepsilon_0(M,M')>0$ such that if $\varepsilon<\varepsilon_0$, then there is a unique solution $u$ of \eqref{INLS} on $I\times \mathbb{R}^N$, with  $u(0)=u_0$,  satisfying
\begin{equation}\label{C} 
\|u-\widetilde{u}\|_{S^1(I)}\lesssim \varepsilon\qquad\textnormal{and}\qquad \|u\|_{S(I)}\lesssim c(M,M').
\end{equation}
\end{lemma}
\begin{proof}
We may assume that $0=\inf I$. First note by using $\|\widetilde{u}\|_{S(I)}\leq \varepsilon$, for some $\varepsilon>0$ enough small, the Strichartz estimates \eqref{SE1}-\eqref{SE2}, Lemma \ref{LC} and a standard continuity argument we have $\|\nabla\widetilde{u}\|_{W(I)}\lesssim M.$

The solution $u$ will be obtained as $u=\widetilde{u}+w$, where $w$ is the solution of the following   IVP 	
\begin{equation}\label{IVPP} 
\begin{cases}
i\partial_tw +\Delta w + H(x,\widetilde{u},w)+e= 0,&  \\
w(0,x)= u_0(x)-\widetilde{u}_0(x),&
\end{cases}
\end{equation}
with $H(x,\widetilde{u},w)=|x|^{-b} \left(|\widetilde{u}+w|^\alpha (\widetilde{u}+w)-|\widetilde{u}|^\alpha \widetilde{u}\right)$. Indeed, let
\begin{equation}\label{IEP} 
G (w)(t):=e^{it\Delta}w_0+i  \int_0^t e^{i(t-s)\Delta}(H(x,\widetilde{u},w)+e)(s)ds
\end{equation}
and define $
B_{\rho}=\{ w\in C(I;\dot{H}^1(\mathbb{R}^N)):\;\|w\|_{S^1(I)}=\|w\|_{S(I)}+\|\nabla w\|_{W(I)}\leq \rho   \}.
$
We show that $G$ is a contraction on $B_{\rho}$.
Similarly as in Proposition \ref{GWPCH1} we deduce
\begin{equation}\label{ST1}
\begin{split}
\|G(w)\|_{S^1(I)}&\lesssim \varepsilon+ \| \nabla H\|_{L^2_IL_x^{\frac{2N}{N+2}}}+\|\nabla e\|_{L^2_IL_x^{\frac{2N}{N+2}}}.
\end{split}
\end{equation}
Let us now estimate $\|\nabla H(\cdot,\widetilde{u},w)\|_{L^2_IL_x^{\frac{2N}{N+2}}}$. From \eqref{SECONDEI} we get
\begin{equation*} 
|\nabla H(x,\widetilde{u},w)| \lesssim |x|^{-b}(|\widetilde{u}|^{\alpha}+|w|^{\alpha})|x|^{-1}|w|+|x|^{-b}(|\widetilde{u}|^\alpha+|w|^\alpha) |\nabla w| +E,
\end{equation*}
 where (since $\alpha\geq 1$)
 $$
 E \lesssim |x|^{-b}\left(|\widetilde{u}|^{\alpha-1}+|w|^{\alpha-1}\right)|w||\nabla \widetilde{u}|.
$$
Therefore, Lemma \ref{LC} and the Hardy inequality lead to
$$
\|\nabla H(\cdot,\widetilde{u},w)\|_{L^2_IL_x^{\frac{2N}{N-2}}} \lesssim \left(\|\nabla \widetilde{u} \|^b_{W(I)}\| \widetilde{u} \|^{\alpha-b}_{S(I)} + \|\nabla w \|^b_{W(I)}\| w\|^{\alpha-b}_{S(I)} \right)\|\nabla w \|_{W(I)}+E_1,
$$
where
\begin{align*}
E_1 \lesssim
\begin{cases}
  \left(\|\tilde u\|^{\alpha-1}_{S(I)}+\|w\|^{\alpha-1}_{S(I)}\right) \|\nabla \tilde u\|_{W(I)}\|\nabla w\|_{W(I)}^b\|w\|_{S(I)}^{1-b},\ \text{if $b<1$};\\
  \left(\|\nabla \tilde u\|_{W(I)}^{b-1}\|\tilde u\|_{S(I)}^{\alpha-b}+\|\nabla w\|_{W(I)}^{b-1}\|w\|_{S(I)}^{\alpha-b}\right)\|\nabla \tilde u\|_{W(I)}\|\nabla w\|_{W(I)},\ \text{if $b\geq1$}.
\end{cases}
\end{align*}
Gathering together the above estimates with our assumptions, we obtain for any $w\in B_{\rho}$,
\begin{align}\label{SP9}\nonumber
&\|\nabla H(\cdot,\widetilde{u},w)\|_{L^2_IL_x^{\frac{2N}{N+2}}}\\
 \lesssim&
\begin{cases}
\left(M^b\varepsilon^{\alpha-b}+\rho^{\alpha}\right)\rho
+(\varepsilon^{\alpha-1}+\rho^{\alpha-1})M\rho,\ \text{if $b<1$};\\
\left(M^b\varepsilon^{\alpha-b}+\rho^{\alpha}\right)\rho+(M^{b-1}\varepsilon^{\alpha-b}+\rho^{\alpha-1})M\rho,\ \text{if $b\geq1$}.
\end{cases}
\end{align}
The relations \eqref{ST1} and \eqref{SP9} yield (choosing $\rho=4C\varepsilon$ and $K=\rho$)
\begin{align*}
\|G(w)\|_{S^1(I)}\leq  \tfrac{\rho}{2}+
  \begin{cases}
\left(M^b\varepsilon^{\alpha-b}+K^b\rho^{\alpha-b}\right)K
+(\varepsilon^{\alpha-1}+\rho^{\alpha-1})MK^b\rho^{1-b},\ \text{if $b<1$};\\
\left(M^b\varepsilon^{\alpha-b}+K^b\rho^{\alpha-b}\right)K+M^{b}K\varepsilon^{\alpha-b}+K^{b}M\rho^{\alpha-b},\ \text{if $b\geq1$}.
\end{cases}
\end{align*}

If $\varepsilon$ is sufficiently small such that
$$
M^b\varepsilon^{\alpha-b}+\rho^{\alpha}<\tfrac{1}{4},\;\;\;M(\varepsilon^{\alpha-1}+\rho^{\alpha-1})<\tfrac{1}{4}\;\;
\textnormal{and}\;\;M\varepsilon^{\alpha-b}+M\rho^{\alpha-1}<\tfrac{1}{4},
$$
then
\begin{equation*}
\|G(w)\|_{S^1(I)}=\|G(w)\|_{S(I)}+\|\nabla G(w)\|_{W(I)}\leq \rho,
\end{equation*}
that is, $G$ is well defined on $B_{\rho}$. Using a similar argument we can also prove that $G$ is a contraction. Therefore, we conclude a unique solution $w$ on $I\times \mathbb{R}^N$ such that
$$
\|w\|_{S^1(I)}\lesssim \varepsilon,
$$
which it turn implies \eqref{C}. This completes the proof.
\end{proof}	

With Lemma \ref{STP} in hand we can show Proposition \ref{stability}.
\begin{proof}[\bf {Proof of Proposition \ref{stability}}] As before, we can assume $0=\inf I$.  Since $\|\widetilde{u}\|_{S(I)}\leq L$,  we may take a partition of $I$ into $n = n(L,\varepsilon)$ intervals $I_j = [t_j ,t_{j+1}]$ such that $\|\widetilde{u}\|_{S(I)}\leq \varepsilon$, where  $\varepsilon<\varepsilon_0(M,2M')$ and $\varepsilon_0$ is given in Lemma \ref{STP}.
Since, on $I_j$,
\begin{equation*}
w(t)=e^{i(t-t_j)\Delta}w(t_j)+i\int_{t_j}^{t}e^{i(t-s)\Delta}(H(x,\widetilde{u},w)+e)(s)ds,
  \end{equation*}
solves the equation in \eqref{IVPP} with initial data $w(t_j)=u(t_j)-\widetilde{u}(t_j)$, by choosing $\varepsilon_1=\varepsilon_1(n,M,M')$ sufficiently small we may reiterate Lemma \ref{STP} to obtain, for each $0\leq j<n$ and $\varepsilon<\varepsilon_1$,
\begin{equation}\label{LP1}
\|u-\widetilde{u}\|_{S^1(I_j)}\leq c(M,M',j)\varepsilon
\end{equation}
and
\begin{equation}\label{LP2}
\|u\|_{S(I_j)}\leq c(M,M',j),
\end{equation}
provided that (for each $0\leq j<n$)
\begin{equation}\label{LP3}
 \|e^{i(t-t_j)\Delta}(u(t_j)-\widetilde{u}(t_j))\|_{S^1(I_j)}\leq c(M,M',j)\varepsilon\leq \varepsilon_0
\end{equation}
and
\begin{equation}\label{LP4}
 \|u(t_j)-\widetilde{u}(t_j)\|_{\dot{H}^1_x}\leq 2M'.
 \end{equation}
By summing \eqref{LP1} and \eqref{LP2} over all subintervals $I_j$, we get the desired.

It remains to establish \eqref{LP3} and \eqref{LP4}. Indeed,
\[
\begin{split}
 \|e^{i(t-t_j)\Delta}w(t_j)\|_{S^1(I_j)}&\lesssim \|e^{it\Delta}w_0\|_{S^1(I_j)}+\| \nabla H(\cdot,\widetilde{u},w)\|_{L^2_{[0,t_j]}L_x^{\frac{2N}{N+2}}}+\|\nabla e\|_{L^2_{I}L_x^{\frac{2N}{N+2}}},
\end{split}
 \]
which by an inductive argument implies\footnote{By \eqref{SP9} and the choice of $\rho$ in Lemma \ref{STP} we get $\|H(\cdot, \widetilde{u},w)\|_{L^2_{[0,t_j]}L_x^{\frac{2N}{N+2}}}\leq c(M,M')(\varepsilon^{\alpha-b+1}+\varepsilon^{\alpha+1}+\varepsilon^{\alpha})$.}
$$
\|e^{i(t-t_j)\Delta}(u(t_j)-\widetilde{u}(t_j))\|_{S^1(I_j)}\lesssim \varepsilon+\sum_{k=0}^{j-1}c(M,M',k)(\varepsilon^{\alpha-b+1}+\varepsilon^{\alpha+1}+\varepsilon^{\alpha}).
$$

Similarly, we also obtain
\[
\begin{split}
\|u(t_j)-\widetilde{u}(t_j)\|_{\dot{H}^1_x}&\lesssim  
M'+\sum_{k=0}^{j-1}C(k,M,M')(\varepsilon^{\alpha-b+1}+\varepsilon^{\alpha+1}+\varepsilon^{\alpha}).
\end{split}
\]
Taking $\varepsilon_1$ sufficiently small, we see that \eqref{LP3} and \eqref{LP4} hold. This completes the proof of the proposition.
 \end{proof}

We end the section with a variational analysis of the equation.

\subsection{Variational analysis}
From Theorem 2.2 and Remark 2.1 of \cite{Yanagida}, we know that the ground state
\[
W(x) = \left(1+\tfrac{|x|^{2-b}}{(N-b)(N-2)}\right)^{-\tfrac{N-2}{2-b}}
\]
solves the elliptic equation
\begin{equation}\label{elliptic}
\Delta W + |x|^{-b} W^{\alpha+1} = 0
\end{equation}
and 
\begin{equation}\label{embedding-est}
\| |x|^{-b}|u|^{\alpha+2}\|_{L^1} \leq C_1 \|\nabla u\|_{L^2}^{\alpha+2}.
\end{equation}

Multiplying \eqref{elliptic} by $Q$ and integrating by parts yields
\[
\| \nabla W\|_{L^2}^2 = \||x|^{-b}W^{\alpha+2}\|_{L^1} =c,
\]
where the final equality follows from direct calculation.  This leads to the following useful identities:
\begin{equation}\label{ywmtbe}
C_1 = c^{-\frac{\alpha}2},\quad E(W) = \tfrac{\alpha c}{2(\alpha+2)}.
\end{equation}

The following lemma shows that solutions to \eqref{INLS} that initially obey the sub-threshold condition
\eqref{threshold} continue to do so throughout their lifespan.  Furthermore, the energy is coercive in this
regime.

\begin{lemma}[Energy trapping]\label{energy-trapping} Suppose that $u:I\times\R^N\to\C$ is a solution to
\eqref{INLS} obeying \eqref{threshold}.  Then there exists $\delta>0$ such that
\begin{equation}\label{quant-below}
\sup_{t\in I}\|u(t)\|_{\dot H^1}<(1-\delta)\|W\|_{\dot H^1}.
\end{equation}
Furthermore,
\begin{equation}\label{energy-sim}
 E(u(t))\sim \|u(t)\|_{\dot H^1}^2\sim \|u_0\|_{\dot H^1}^2\qtq{for all}t\in I.
\end{equation}
\end{lemma}

\begin{proof} We first suppose $u(t)\in \dot H^1$ satisfies
\begin{equation}\label{f-is-below}
E(u_0) < (1-\delta_0)E(W) = (1-\delta_0)\tfrac{\alpha c}{2(\alpha+2)} \qtq{and} \|u_0\|_{\dot H^1}^2 \leq
\|W\|_{\dot H^1}^2 = c.
\end{equation}
for some $\delta_0>0$.   By the sharp inequality \eqref{embedding-est} and \eqref{ywmtbe}, we may also write
\begin{equation}\label{energy-lb}
E(u(t)) \geq \tfrac12 \|u(t)\|_{\dot H^1}^2 - \tfrac{c^{-\frac{\alpha}2}}{\alpha+2}\|u(t)\|_{\dot
H^1}^{\alpha+2}.
\end{equation}
Thus, with $y=\|u(t)\|_{\dot H^1}^2$ and
\[
F(y)=\frac{y}2-\tfrac{c^{-\frac{\alpha}2}}{\alpha+2}y^{\frac{\alpha+2}{2}}, \qtq{while} y\leq c,
\]
we have $F(0)<(1-\delta_0)\tfrac{\alpha c}{2(\alpha+2)}$. Using this, continuity of the flow in $\dot H^1$, and
conservation of energy, we deduce that $y<c-\delta_1$ for some $\delta_1$. Then  \eqref{quant-below} follows
directly.

For \eqref{energy-sim}, it suffices (by the conservation of energy) to show that if $f\in \dot H^1$ is as
above, then $E(f)\gtrsim \|f\|_{\dot H^1}^2$.  In fact, using $\|f\|_{\dot H^1}^2\leq c$ and \eqref{energy-lb},
we have
\[
E(f) \geq \|f\|_{\dot H^1}^2\{\tfrac12 - \tfrac{c^{-\frac{\alpha}2}}{\alpha+2}\cdot c^{\frac{\alpha}2}\}\geq
\tfrac{\alpha}{2(\alpha+2)}\|f\|_{\dot H^1}^2.
\]
\end{proof}

\begin{remark}\label{proof od corollary} Note that, to prove Corollary \ref{corollary}, we use Lemma \ref{energy-trapping}. Specifically, under the given hypothesis, we obtain \eqref{quant-below}, and by applying Theorem \ref{T}, we establish that the solution of \eqref{INLS} is global and scatters in $\dot{H}^1$.  
\end{remark}

Finally, the next lemma shows that the sub-threshold assumption in \eqref{threshold} yields the coercivity needed to run the virial argument. In particular, the quantity appearing in the lemma is precisely the functional that shows up in the time derivative of the virial quantity.

\begin{lemma}[Coercivity]\label{L:coercive} Suppose that $u:I\times\R^N\to\C$ is a solution to \eqref{INLS}. Suppose $\sup_{t\in I}\|\nabla u(t)\|_{L^2}\leq (1-\delta_0)\|\nabla Q\|_{L^2}$ for some $\delta_0>0$, then there exists $\delta>0$ such that
\[
\int |\nabla u(t,x)|^2 - |x|^{-b}|u(t,x)|^{\alpha+2}\,dx \geq \delta \int |\nabla u(t,x)|^2\,dx \qtq{uniformly over}t\in I.
\]
\end{lemma}

\begin{proof} The sharp inequality \eqref{embedding-est} and the hypothesis imply
$$
\|u(t)\|_{\dot H^1}^2 - \||x|^{-b} |f|^{\alpha+2}\|_{L^1}\geq \|u(t)\|_{\dot H^1}^2 - C_1\|u(t)\|^{\alpha+2}_{\dot{H}^1}\geq \left(1-(1-\delta_0)^\alpha C_1\|\nabla Q\|^{\frac{\alpha}{2}}_{L^2}\right)\|u(t)\|_{\dot H^1}^2.
$$
We obtain the desired result using the relation \eqref{ywmtbe}.


\end{proof}

\section{Existence of minimal non-scattering solution}\label{S:exist}

This section contains several results related to constructing a minimal non-scattering solution (or critical solution), as detailed in Proposition \ref{MS} below. In particular, it includes two essential new ingredients of our work: Proposition~\ref{P:embed} and Proposition \ref{PS}. To do that, we start with the following linear profile decomposition of \cite{Keraani} (see \cite[Theorem~4.1]{Visan}).	

\begin{proposition}[Linear profile decomposition]\label{P:LPD}
Let $u_n$ be a bounded sequence in $\dot H^1(\R^N)$. Then the following holds up to a subsequence:

There exist $J^*\in\mathbb{N}\cup\{\infty\}$; profiles $\phi^j\in \dot H^1\backslash\{0\}$; scales $\lambda_n^j\in(0,\infty)$; space translation parameters $x_n^j\in\R^N$; time translation parameters $t_n^j$; and remainders $w_n^J$ so that writing
\[
g_n^j f(x) = (\lambda_n^j)^{-\frac{N-2}2} f(\tfrac{x-x_n^j}{\lambda_n^j}),
\]
we have the following decomposition for $1\leq J\leq J^*$:
\[
u_n = \sum_{j=1}^J g_n^j[e^{it_n^j\Delta}\phi^j] + w_n^J.
\]
This decomposition satisfies the following conditions:
\begin{itemize}
\item Energy decoupling: writing $P(u)=\| |x|^{-b} |u|^{\alpha+2}\|_{L^1}$, we have
\begin{equation}\label{energy-decoupling}
\begin{aligned}
&\lim_{n\to\infty} \bigl\{ \|\nabla u_n\|_{L^2}^2 - \sum_{j=1}^J \|\nabla \phi^j \|_{L^2}^2 - \|\nabla w_n^J\|_{L^2}^2\bigr\} = 0, \\
&\lim_{n\to\infty} \bigl\{P(u_n) - \sum_{j=1}^J P(e^{it_n^j\Delta}\phi^j)-P(w_n^J)\bigr\} = 0.
\end{aligned}
\end{equation}
\item Asymptotic vanishing of remainders:
\begin{equation}\label{vanishing}
\limsup_{J\to J^*}\limsup_{n\to\infty} \|e^{it\Delta}w_n^J\|_{L_{t,x}^{\frac{2(N+2)}{N-2}}(\R\times\R^N)} = 0.
\end{equation}
\item Asymptotic orthogonality of parameters: for $j\neq k$,
\begin{equation}\label{orthogonality}
\lim_{n\to\infty} \biggl\{\log\bigl[\tfrac{\lambda_n^j}{\lambda_n^k}\bigr] + \tfrac{|x_n^j-x_n^k|^2}{\lambda_n^j\lambda_n^k} + \tfrac{|t_n^j(\lambda_n^j)^2-t_n^k(\lambda_n^k)^2|}{\lambda_n^j\lambda_n^k}\biggr\} = \infty.
\end{equation}
\end{itemize}
In addition, we may assume that either $t_n^j\equiv 0$ or $t_n^j\to\pm\infty$, and that either $x_n^j\equiv 0$ or $|x_n^j|\to\infty$.
\end{proposition}


The following proposition is the first essential ingredient for proving the main theorem. It establishes scattering solutions to \eqref{INLS} associated with initial data living sufficiently far from the origin. This result allows us to extend the construction of a minimal blowup solution from the radial to the non-radial setting; moreover, it guarantees that the compact solutions we construct remain localized near the origin, which facilitates the use of the localized virial argument.

Before starting the result, let us introduce the following notations:  For $x_n\in\R^N$ and $\lambda_n\in(0,\infty)$, we define
\[
g_n\phi(x) = \lambda_n^{-\frac{N-2}2}\phi(\tfrac{x-x_n}{\lambda_n}).
\]
Observe that
\begin{equation}\label{gc}
g_n e^{it\Delta} = e^{i\lambda_n^2 t\Delta}g_n.
\end{equation}

\begin{proposition}\label{P:embed} Let $\lambda_n \in (0,\infty)$, $x_n\in \R^N$, and $t_n\in\R$ satisfy
\[
\lim_{n\to\infty} \tfrac{|x_n|}{\lambda_n} = \infty \qtq{and} t_n\equiv 0 \qtq{or} t_n\to\pm\infty.
\]
Let $\phi\in \dot H^1$ and define
\[
\phi_n(x) = g_n [e^{it_n\Delta}\phi] = \lambda_n^{-\frac{N-2}2}[e^{it_n\Delta}\phi](\tfrac{x-x_n}{\lambda_n}).
\]
Then for all $n$ sufficiently large, there exists a global solution $v_n$ to \eqref{INLS} satisfying
\[
v_n(0) = \phi_n \qtq{and}\|v_n\|_{ S(\R)} \lesssim 1,
\]
with\footnote{Recalling the norm $\|\cdot\|_{S(\R)}$ is given in \eqref{scatteringsize}.} implicit constant depending only on $\|\phi\|_{\dot H^1}$.

Furthermore, for any $\eps>0$ there exists $N\in\mathbb{N}$ and $\psi\in C_c^\infty(\R\times\R^N)$ so that for $n\geq N$, we have
\[
\|\lambda_n^{\frac{N-2}2} v_n(\lambda_n^2(t-t_n),\lambda_n x+x_n)-\psi\|_{S(\R)} < \eps.
\]

\end{proposition}

\begin{proof} The proof is based on the construction of suitable approximate solutions to \eqref{INLS} with initial condition asymptotically matching $\phi_n$.  Our approximation is based on solutions to the \emph{linear} Schr\"odinger equation.

To define our approximation, we let $\theta\in(0,1)$ be a small parameter to be determined below and introduce a frequency cutoff $P_n$ and spatial cutoff $\chi_n$ as follows.  First, we let
\[
P_n = P_{|\frac{x_n}{\lambda_n}|^{-\theta}\leq\cdot\leq |\frac{x_n}{\lambda_n}|^{\theta}}.
\]
Next, we take $\chi_n$ to be a smooth function satisfying
\[
\chi_n(x) = \begin{cases} 1 & |x+\tfrac{x_n}{\lambda_n}| \geq \tfrac12 |\tfrac{x_n}{\lambda_n}| \\ 0 & |x+\tfrac{x_n}{\lambda_n}| < \tfrac14 |\tfrac{x_n}{\lambda_n}|,\end{cases}
\]
with $|\partial^\alpha \chi_n|\lesssim |\tfrac{x_n}{\lambda_n}|^{-|\alpha|}$ for all multiindices $\alpha$.  Observe that $\chi_n\to 1$ pointwise as $n\to\infty$.

We now define approximations $\tilde v_{n,T}$ as follows:

\begin{definition}[Approximate solutions] First, let
\[
I_{n,T} := [a_{n,T}^-,a_{n,T}^+]:= [-\lambda_n^2 t_n - \lambda_n^2 T,-\lambda_n^2t_n+\lambda_n^2T]
\]
and for $t\in I_{n,T}$ define
\begin{align*}
\tilde v_{n,T}(t) & = g_n\left[\chi_n P_n e^{i(\lambda_n^{-2}t+t_n)\Delta}\phi\right] \\
& = \chi_n(\tfrac{\cdot-x_n}{\lambda_n})e^{i(t+\lambda_n^2 t_n)\Delta} g_n[P_n\phi].
\end{align*}

Next, let
\[
I_{n,T}^+ := (a_{n,T}^+,\infty),\quad I_{n,T}^- = (-\infty,a_{n,T}^-)
\]
and set
\[
\tilde v_{n,T}(t) =\begin{cases} e^{i(t-a_{n,T}^+)\Delta} [\tilde v_{n,T}(a_{n,T}^+)] & t\in I_{n,T}^+ \\ e^{i(t-a_{n,T}^-)\Delta}[\tilde v_{n,T}(a_{n,T}^-)]& t\in I_{n,T}^-.\end{cases}
\]
\end{definition}

We will establish the existence of the solutions $v_n$ by applying the stability result (Proposition~\ref{stability}). To this end, we need to show three conditions.

\textbf{Condition 1.}
\begin{equation}\label{vnt-bounds}
\limsup_{T\to\infty}\limsup_{n\to\infty}\bigl\{\|\tilde v_{n,T}\|_{L_t^\infty \dot H_x^1} + \|\tilde v_{n,T}\|_{S(\R)} \bigr\}\lesssim 1.
\end{equation}

\begin{proof} We estimate separately on $I_{n,T}$ and $I_{n,T}^{\pm}$.

We first estimate on $I_{n,T}$. Noting that
$
\|\chi_n\|_{L^\infty} + \|\nabla[\chi_n]\|_{L^N} \lesssim 1.
$
Combining the product rule, H\"older's inequality and the Sobolev embedding $\dot H^1\hookrightarrow L^{\frac{2N}{N-2}}$, we deduce that $\chi_n:\dot H^1\to \dot H^1$ boundedly. Then the $L_t^\infty \dot H_x^1$ bound on $I_{n,T}$ follows from the Strichartz estimate. The $\dot S^1$ bound on $I_{n,T}$ also follows from Sobolev embedding and Strichartz. In the same way, we obtain the estimate on $I_{n,T}^{\pm}$. 
\end{proof}

\textbf{Condition 2.}
\[
\lim_{T\to\infty}\limsup_{n\to\infty}\|\tilde v_{n,T}(0)-\phi_n\|_{\dot H^1} = 0.
\]

\begin{proof} We treat two cases:
\begin{itemize}
\item[(i)] $t_n\equiv 0$ (so that $0\in I_{n,T}$),
\item[(ii)] $t_n\to+\infty$ (so that $0\in I_{n,T}^+$ for all $n$ sufficiently large). The case $t_n\to-\infty$ is similar.
\end{itemize}
Indeed, Case (i). If $t_n\equiv 0$, then $
\| \tilde v_{n,T}(0)-\phi_n\|_{\dot H^1} = \|(\chi_n P_n-1)\phi\|_{\dot H^1} \to 0
$
as $n\to\infty$, using the dominated convergence theorem.

Case (ii) If $t_n\to\infty$, then applying \eqref{gc} and by construction, it follows that
\begin{align*}
\tilde v_{n,T}(0) & = e^{-ia_{n,T}^+\Delta}g_n[\chi_n P_n e^{i(\lambda_n^{-2}a_{n,T}^++t_n)\Delta}\phi] \\
&  = g_n e^{it_n\Delta} [e^{-iT\Delta}\chi_n P_n e^{iT\Delta}\phi].
\end{align*}
Thus,
\begin{align*}
\| \tilde v_{n,T}(0) - \phi_n \|_{\dot H^1} & = \|[e^{-iT\Delta}\chi_n P_n e^{iT\Delta}-1]\phi\|_{\dot H^1}\\
& = \| [\chi_n P_n -1]e^{iT\Delta}\phi \|_{\dot H^1} \to 0\;\;\;\textnormal{as}\;\;\; n\to\infty.
\end{align*}
\end{proof}

\textbf{Condition 3.} Defining
\[
e_{n,T}=(i\partial_t + \Delta)\tilde v_{n,T} + |x|^{-1}|\tilde v_{n,T}|^2 \tilde v_{n,T},
\]
we have
\[
\lim_{T\to\infty}\limsup_{n\to\infty}\|\nabla e_{n,T} \|_{S'(L^2;\R)} = 0.
\]

\begin{proof} We estimate on $I_{n,T}$ and $I_{n,T}^\pm$.

We first treat the interval $I_{n,T}$.  We write $e_{n,T}=e_{n,T}^{\text{lin}}+e_{n,T}^{\text{nl}}$, with
\begin{align*}
e_{n,T}^{\text{lin}} & = \Delta[\chi_n(\tfrac{x-x_n}{\lambda_n})] e^{i(t+\lambda_n^2 t_n)\Delta}g_n[P_n\phi] \\
& \quad + 2\nabla[\chi_n(\tfrac{x-x_n}{\lambda_n})]\cdot e^{i(t+\lambda_n^2 t_n)\Delta}\nabla[g_n P_n\phi]
\end{align*}
and
\[
e_{n,T}^{\text{nl}} = \lambda_n^{-\frac{(N-2)\alpha}{2}} g_n\bigl\{ |\lambda_n x+ x_n|^{-b} \chi_n^{\alpha+1} |\Phi_n|^{\alpha}\Phi_n\bigr\},
\]
where
\[
\Phi_n(t,x) = P_n e^{i(\lambda_n^{-2} t+t_n)\Delta} \phi.
\]

When we apply the gradient to $e_{n,T}^{\text{lin}}$, we obtain a sum of terms of the form
\[
\partial^j[\chi_n(\tfrac{x-x_n}{\lambda_n})] \cdot e^{i(t+\lambda_n^2 t_n)\Delta} \partial^{3-j}[g_n P_n \phi]\qtq{for}j\in\{1,2,3\},
\]
where $\partial$ denotes a derivative in $x$. Hence, by H\"older's inequality, Bernstein's inequality, and the fact that $j\geq 1$, one has
\begin{align*}
\| \partial^j&[\chi_n(\tfrac{x-x_n}{\lambda_n})] \cdot e^{i(t+\lambda_n^2 t_n)\Delta}\partial^{3-j}[g_n P_n\phi]\|_{L_t^1 L_x^{2}(I_{n,T}\times\R^N)} \\
& \lesssim |I_{n,T}|\, \|\partial^j[\chi_n(\tfrac{x-x_n}{\lambda_n})]\|_{L_x^\infty} \|\partial^{3-j}[g_nP_n\phi]\|_{L_t^\infty L_x^2} \\
& \lesssim \lambda_n^2T\cdot\lambda_n^{-j}|\tfrac{x_n}{\lambda_n}|^{-j}\cdot\lambda_n^{j-2}\|\partial^{3-j}P_n\phi\|_{L_t^\infty L_x^2} \\
& \lesssim T|\tfrac{x_n}{\lambda_n}|^{-j}|\tfrac{x_n}{\lambda_n}|^{|2-j|\theta}\to 0\qtq{as}n\to\infty\;,\;\textnormal{for}\;\;\theta\;\; \textnormal{small}.
\end{align*}

Estimating $e^{\text{nl}}_{n,T}$ on $I_{n,T}$. A change of variables and H\"older's inequality implies that
\[
\|\nabla e_{n,T}^{\text{nl}}\|_{L_t^2 L_x^{\frac{2N}{N+2}}(I_{n,T}\times\R^N)}  \leq \lambda_n^{b}T^{\frac12}\| \nabla\bigl[ |\lambda_n x+ x_n|^{-b} \chi_n^{\alpha+1}|\Phi_n|^{\alpha}\Phi_n\bigr]\|_{L_t^\infty L_x^{\frac{2N}{N+2}}(I_{n,T}\times\R^N)}.
\]
We next note that
\[
\|\partial^j\bigl[\lambda_n x+x_n|^{-b}\chi_n^{\alpha+1}\bigr]\|_{L_x^\infty} \lesssim |\tfrac{x_n}{\lambda_n}|^{-j}|x_n|^{-b},\quad j\in\{0,1\}.
\]
So, applying the product rule, H\"older's inequality, Sobolev embedding, and Bernstein, we get
\begin{align*}
\| \nabla e_{n,T}^{\text{nl}}\|_{L_t^2 L_x^{\frac{2N}{N+2}}(I_{n,T}\times\R^N)} &\lesssim T^{\frac12}|\tfrac{x_n}{\lambda_n}|^{-b}\|\Phi_n\|_{L_t^\infty L_x^{N\alpha}}^{\alpha}  \sum_{j=0}^1 |\tfrac{x_n}{\lambda_n}|^{-j} \|\partial^{1-j}\Phi_n\|_{L_t^\infty L_x^{2}} \\
& \lesssim T^{\frac12}|\tfrac{x_n}{\lambda_n}|^{-1+\theta(5/2-3/r)}\|\phi\|_{\dot H^1}^{\alpha+1} \to 0 \qtq{as}n\to\infty.
\end{align*}

Let's focus on the intervals $I_{n,T}^\pm$. We'll only address $I_{n,T}^\pm$ since a similar argument applies to the remaining interval.
\[
e_{n,T} = |x|^{-b} |V_{n,T}|^{\alpha} V_{n,T},
\]
where
\[
V_{n,T}= e^{i(t-a_{n,T}^+)\Delta}[\tilde v_{n,T}(a_{n,T}^+)].
\]
Estimating as in Lemma~\ref{LC} and using \eqref{vnt-bounds}, we deduce
\begin{align*}
\| \nabla e_{n,T}\|_{L_t^2 L_x^{\frac{2N}{N-2}}(I_{n,T}^+\times\R^N)} & \lesssim \| V_{n,T}\|_{L_{t,x}^{\frac{2(N+2)}{N-2}}(I_{n,T}^+\times\R^N)}^{\alpha-b}\|\nabla V_{n,T}\|_{L_t^p L_x^{r}(I_{n,T}^+\times\R^N)}^{1+b}\\
& \lesssim \|\phi\|_{\dot H^1}^{1+b}\|e^{it\Delta}[\tilde v_{n,T}(a_{n,T}^+)]\|_{L_{t,x}^{\frac{2(N+2)}{N-2}}(\R_+\times\R^N)}^{\alpha-b}.
\end{align*}
We now use the definition of $\tilde v_{n,T}$ and \eqref{gc} to write
\[
e^{it\Delta}\tilde v_{n,T}(a_{n,T}^+) = e^{it\Delta} g_n [\chi_n P_n e^{iT\Delta}\phi] = g_n e^{i\lambda_n^{-2}t\Delta}[\chi_n P_n e^{iT\Delta}\phi].
\]
Thus by a change of variables, Sobolev embedding, and Strichartz, we have
\begin{align}
\|e^{it\Delta}&[\tilde v_{n,T}(a_{n,T}^+)]\|_{L_{t,x}^{\frac{2(N+2)}{N-2}}(\R_+\times\R^N)}\nonumber \\
& = \|e^{it\Delta}\chi_n P_n e^{iT\Delta}\phi \|_{L_{t,x}^{\frac{2(N+2)}{N-2}}(\R_+\times\R^N)}\nonumber \\
& \lesssim \|\nabla[\chi_n P_n -1]e^{iT\Delta}\phi\|_{L_x^2} + \| e^{it\Delta}\phi\|_{L_{t,x}^{\frac{2(N+2)}{N-2}}([T,\infty)\times\R^N)}.\label{l-e-t}
\end{align}
Noting that the first term in \eqref{l-e-t} tends to zero as $n\to\infty$ by the dominated convergence theorem, and the second term in \eqref{l-e-t} tends to zero as $T\to\infty$ by Strichartz and the monotone convergence theorem.
\end{proof}

With conditions 1--3 in place, we apply Proposition~\ref{stability} 
to obtain the existence of solutions $v_n$ to \eqref{INLS} initial data $\phi_n$ and finite $S(\R)$ norm. Moreover,
\begin{equation}\label{embed-cc1}
\limsup_{T\to\infty}\lim_{n\to\infty} \|v_n-\tilde v_{n,T}\|_{S(\R)} = 0.
\end{equation}

We now adapt the arguments from \cite{KMVZZ} to the approximation by $C_c^\infty$ functions. We only give the sketch of the proof. Let $\eps>0$ and by \eqref{embed-cc1}, it suffices to show that 
\[
\|\lambda_n^{\frac{N-2}2} \tilde v_{n,T}(\lambda_n^2(t-t_n),\lambda_n x+x_n)-\psi(t,x)\|_{ S(\R)} \lesssim \eps
\]
for all $n,T$ sufficiently large.  In fact, we will see that $\psi$ will be taken to be an approximation to $e^{it\Delta}\phi$.

Observe that for $t\in(-T,T)$, we have
\[
\lambda_n^{\frac{N-2}2} \tilde v_{n,T}(\lambda_n^2(t-t_n),\lambda_n x+x_n) = \chi_n(x) e^{it\Delta}P_n\phi(x).
\]
We can approximate this in $S(\R)$ by $e^{it\Delta}\phi$ for large $n$, which can in turn be well-approximated by a compactly supported function of space-time on $[-T,T]\times\R^N$.  For $t>T$ (the case $t<T$ being treated similarly)
\begin{align*}
\lambda_n^{\frac{N-2}2}\tilde v_{n,T}(\lambda_n^2(t-t_n),\lambda_n x+x_n) & = g_n^{-1} e^{i\lambda_n^2(t-T)\Delta}g_n \chi_n e^{iT\Delta}P_n \phi \\
& = e^{it\Delta}[e^{-iT\Delta}\chi_n e^{iT\Delta}]P_n \phi.
\end{align*}
Again, we can approximate this in $S(\R)$ by $e^{it\Delta}\phi$ for large $n$, which is again well-approximated by a compactly supported function of space-time on $[T,\infty)\times\R^N$.
\end{proof}

Next, we show the main ingredient of our work, the Palais-Smale Condition. The proof is technical, combining profile decomposition with the first main ingredient to construct the nonlinear profiles. Then, two lemmas are crucial to complete the proof: the existence of a bad profile and the decoupling of kinetic energy.

\begin{proposition}{(\bf Palais-Smale Condition)}\label{PS}. Let $u_n: I_n\times \R^N \rightarrow \C$ be a sequence of maximal life-span solutions to \eqref{INLS} and $t_n \in I_n$. Suppose that
\begin{equation}\label{Hyp1}
\lim_{n\rightarrow\infty}sup_{t\in I_n}\|u_n(t)\|^2_{\dot{H}^1}= Kc < \|W\|^2_{\dot{H}^1}
\end{equation}
$$
\lim_{n\rightarrow \infty} \|u_n\|_{S(t\geq t_n)} = \lim_{n\rightarrow \infty} \|u_n\|_{S(t\leq t_n)} = \infty,
$$
then there exists $\{\lambda_n\}\subset \R^{+}$ such that $\{\lambda_n^{-\frac{N-2}{2}}u_n(t_n,\frac{x}{\lambda_n})\}$ is precompact in $\dot{H}^1$.
\end{proposition}
\begin{proof}
By the time translation invariance, we may assume that $t_n \equiv 0$, then
$$
\lim_{n\rightarrow \infty} \|u_n\|_{S(t\geq 0)} = \lim_{n\rightarrow \infty} \|u_n\|_{S(t\leq 0)} = \infty.
$$
Given that $u_{n,0} = u_n(0)$ is bounded \eqref{Hyp1}, we apply the linear profile decomposition, possibly passing to a subsequence
\begin{align}
u_n(0)=\sum_{j=1}^{J}g_n^j[e^{it_n^j\Delta}\phi^j]+w_n^J.
\end{align}

We construct solutions (nonlinear profiles) to INLS: fixing $j$, if $\big|\frac{x_n^j}{\lambda_n^j}\big|\to \infty $ up to a subsequence, by Proposition \ref{P:embed}, there exists a global solution $v_n^j$ to \eqref{INLS} with initial data $g_n^j[e^{it_n^j\Delta}\phi^j]$ and $\|v_n^j\|_{S(\R)}<\infty$. If $\big|\frac{x_n^j}{\lambda_n^j}\big|$ tends to finite number, we can always assume $x_n^j\equiv0$ in this case. For the setting $x_n^j\equiv0$, following the argument of \cite{KM}, we can construct the nonlinear profile $v^j$ associate to $(\phi^j,\{t_n^j\})$ such that
\begin{equation}\label{PS00}
\|v^j(t_n^j)-e^{it_n^j\Delta}\phi^j\|_{\dot{H}^1}\to 0,\ \ n\to\infty. 
\end{equation}
In that case, we define $v^j_n(t, x) = \lambda_n^{-\frac{N-2}{2}}v^j\left( \frac{t}{(\lambda_n^j)^2} +t_n^j, \frac{x}{\lambda_n^j}  \right).$ Note that, $v_n^j(0)=g_n^jv^j(t_n^j)$. 

On the other hand, the relation \eqref{energy-decoupling} and small data theory, implies that there exists $J_0 \geq 1$ such that for all $j \geq J_0$, $\|\phi^j\|_{\dot{H}^1}$ small and the solutions $v^j_n$
are global with
\begin{equation}\label{PSC1}
\|v^j_n\|_{L^\infty\dot{H}^1(\R)}+\|v_n^j\|_{S(\R)}\lesssim \|\phi^j\|_{\dot{H}^1}. 
\end{equation}
Note that, for $j < J_0$ there exists at least one “bad” nonlinear profile such that
\begin{equation}\label{PSC2}
 \|v_n^j\|^2_{L^\infty_t\dot{H}^1({I^j_n})}\geq K_c . 
\end{equation}
Indeed, If $\|v_n^j\|^2_{L^\infty_t\dot{H}^1({I^j_n})}< K_c$ for all $j$, the definition of $K_c$ implies that $v_n^j$ exists globally and $\|v^j_n\|_{S(\R)}$ is bounded, contradicting the lemma below.
\begin{lemma}{(\bf At least one bad profile)}\label{bp} There exists $1 \leq j_0 < J_0$ such that
\begin{equation}\label{bp1}
\| v_n^{j_0}\|_{S([0,T_{n,j_0}))} = \infty.    
\end{equation}
\end{lemma}
\begin{proof}
The proof follows by contradiction: if \eqref{bp1} does not hold, then $T_{n,j_0}= \infty$ and for all $1 \leq j < J_0$,
\begin{equation}\label{bp00}
\limsup_{n \rightarrow \infty} \| v_n^{j}\|_{S([0,\infty))}
< \infty.
\end{equation}
Subdividing $[0, \infty)$ into intervals where the scattering size of $v^j_n$ is small, applying the Strichartz inequality on each such interval, and then summing,
we obtain
\begin{equation}\label{bp01}
\|\nabla v_n^{j}\|_{W([0,\infty))}<\infty.    
\end{equation}
Combining \eqref{PSC1}, \eqref{bp00} and \eqref{Hyp1}, we deduce (for $n$ sufficiently large)
\begin{equation*}
   \sum_{j\geq 1} \| v_n^{j}\|_{S([0,\infty))}\lesssim 1+ \sum_{j\geq J_0}\|\nabla \phi^j\|_{L^2}^2\lesssim 1+E_c.
\end{equation*}
From these assumptions, we'll establish a bound on the forward-in-time scattering size of $u_n$, thus leading to a contradiction. Indeed, define
\[
u_n^{J} = \sum_{j=1}^{J} v_n^j+e^{it\Delta}w_n^J.
\]
If the following conditions hold:
\begin{align}
& \limsup_{n\to\infty}\bigl\{ \|u_n^{J}\|_{L_t^\infty \dot H_x^1} + \|u_n^{J}\|_{S(\R)}\bigr\}\lesssim 1, \label{unJbds} \\
&\limsup_{n\to\infty}\|u_n(0)-u_n^{J}\|_{\dot{H}^1}=0,\label{unJvan}\\
& \limsup_{n\to\infty} \|\nabla[(i\partial_t + \Delta)u_n^{J} + |x|^{-b}|u_n^{J}|^{\alpha+1} u_n^{J}] \|_{L_t^2 L_x^{\frac{2N}{N+2}}}\lesssim \varepsilon, \label{unJapprox}
\end{align}
then the stability result (Proposition \ref{stability}) shows that $\|u_n\|_{S([0,\infty))}<\infty$, which contradicts the fact that $\|u_n\|_{S\left([0,\infty)\right)}\to \infty$ as $n\to\infty$.

\

We first show \eqref{unJbds}. Note that
\begin{align*}
\Big|\sum_{j=1}^Jv_n^j\Big|^2=\sum_{j=1}^J|v_n^j|^2+\sum_{j\neq k}v_n^jv_n^k,
\end{align*}
then taking the $L_{t}^{\frac{\Bar{r}}{2}}L_x^{\frac{\Bar{r}}{2}}$ norm in both side, we can get
\begin{align*}
 \Big\|(\sum_{j=1}^{J}v_n^j)^2\Big\|_{L_{t}^{\frac{\bar{r}}{2}}L_x^{\frac{\bar{r}}{2}}}\leq
\sum_{j=1}^{J}\|v_n^j\|_{S([0,\infty))}^2+\sum_{j\neq k}\|v_n^jv_n^k\|_{L_{t}^{\frac{\bar{r}}{2}}L_x^{\frac{\bar{r}}{2}}}.
\end{align*}
Since $\lim_{n\to\infty}\|v_n^jv_n^k\|_{L_{t}^{\frac{\bar{r}}{2}}L_x^{\frac{\bar{r}}{2}}}=0$ (by the orthogonality conditions \eqref{orthogonality}) we have
\begin{align*}
\limsup_{n\to\infty}\|u_n^J\|\leq \limsup_{n\to\infty}\sum_{j=1}^J\|v_n^{j}\|_{S([0,\infty))}\lesssim 1+E_c,
\end{align*}
which is independent of $J$. Moreover, $\|e^{it\Delta}w_n^J\|_{S\left([0,\infty)\right)}\lesssim \|w_n^J\|_{\dot{H}^1}$ is finite. Therefore, the first estimate holds. In the same way (using \eqref{bp01}), we also obtain that $\|u_n^J\|_{L^\infty \dot{H}^1}$ is bounded.

\ We now consider \eqref{unJvan}. Using the fact that $v_n^j(0)=g_n^jv^j(t_n^j)$
$$
\|u_{0,n}-u_{n}^{J}(0)\|_{\dot{H}^1}\lesssim\sum_{j=1}^J \left\|\left( g_n[e^{itj_n\Delta}\phi^j]-g_n^jv^j(t_n^j)\right)\right\|_{\dot{H}^1}\lesssim \sum_{j=1}^J \| e^{itj_n\Delta}\phi^j-v^j(t_n^j)\|_{\dot{H}^1},
$$
which goes to zero, as $n\rightarrow \infty$, by \eqref{PS00}.

\ Finally, we turn to show \eqref{unJapprox}. To this end, we write $f(z)=|z|^{\alpha} z$ and observe
\begin{align}\label{enJ1}
e_n=(i\partial_t + \Delta)u_n^J + |x|^{-b} f(u_n^J)=&|x|^{-b}\left[f\bigl(\sum v_n^j\bigr) - \sum f(v_n^j)\right].
\end{align}
Given \eqref{SECONDEI}, to estimate $\nabla e_n$, it is sufficient to estimate terms of the following types:
\begin{itemize}
\item[(1)] $ |x|^{-b} |v_n^j|^{\alpha-1}\cdot v_n^\ell \cdot T v_n^k $,\quad $T\in\{|x|^{-1},\nabla\}$,
\item[(2)] $ |x|^{-b} |e^{it\Delta}w_n^{J}|^{\alpha} \cdot T e^{it\Delta}w_n^{J} $,\quad $T\in\{|x|^{-1},\nabla\}$,
\item[(3)] $ |x|^{-b} |e^{it\Delta}w_n^{J}|^{\alpha} \cdot \nabla u_n^{J}$,
\item[(4)] $ |x|^{-b} |u_n^{J}|^{\alpha-1}\cdot e^{it\Delta}w_n^{J} \cdot \nabla u_n^{J}$,
\item[(5)] $ |x|^{-b} |u_n^{J}|^{\alpha} \cdot T e^{it\Delta}w_n^{J} $,\quad $T\in\{|x|^{-1},\nabla\}$,
\end{itemize}
where we have $j\neq k$ and $\ell\in\{1,\dots,J\}$.  The first term also involves a constant $C_{J}$ that grows with $J$.  However, we will shortly see that for each fixed $J$, these terms tend to zero in $L_t^2 L_x^{\frac{2N}{N+2}}$ as $n\to\infty$, so that this constant is ultimately harmless.

Using the spaces appearing in Lemma~\ref{LC}, it follows that
\begin{align*}
\|(1)\|_{L_t^2 L_x^{\frac{2N}{N+2}}}\lesssim
\begin{cases}
  \||v_n^k|^{\alpha-b}Tv_n^j\|_{L_t^{q_1}L_x^{r_1}}\|\nabla v_n^\ell\|_{L_t^{q_0}L_x^{r_0}}\|\nabla v_n^k\|_{L_t^{q_0}L_x^{r_0}}^{b-1}\;\; b\geq1;\\
  \||v_n^k|^{\alpha-1}Tv_n^j\|_{L_t^{q_2}L_x^{r_2}}\|\nabla v_n^\ell\|_{L_t^{q_0}L_x^{r_0}}^b\|v_n^\ell\|_{S\left([0,\infty)\right)}^{1-b}\ \ \ \ \ b<1;
\end{cases}
\end{align*}
where $(q_1,r_1)$ and $(q_2,r_2)$ satisfy
$$\frac{N+2}{2N}=\frac1{r_1}+\frac{b}{r_0},\ \ \frac{1}{2}=\frac1{q_1}+\frac{b}{q_0}$$
and
$$\frac{N+2}{2N}=\frac1{r_2}+\frac{b}{r_0}+\frac{1-b}{\bar r},\ \ \frac{1}{2}=\frac1{q_2}+\frac{b}{q_0}+\frac{1-b}{\bar q}.$$
As the norms $\|v_n^\ell\|_{S([0,\infty))}$ and $\|\nabla v_n^\ell\|_{L_t^{q_0}L_x^{r_0}}$ are bounded. It suffices to show that
\begin{align}\label{np-decouple}
  \lim_{n\to\infty}\||v_n^k|^{\alpha-b}Tv_n^j\|_{L_t^{q_1}L_x^{r_1}}=0\qtq{for}j\neq k.
\end{align}
In fact, it follows from approximation by functions in $C_c^\infty(\R\times\R^d)$ and the use of the orthogonality conditions \eqref{orthogonality} (see e.g. \cite{Keraani} or \cite[Lemma~7.3]{Visan}).

 Next, we turn to estimate terms $(2),(3)$. Choosing $J>J'$ large enough such that
 $$\lim_{n\to\infty}\|e^{it\Delta}w_n^{J}\|_{S\left([0,\infty)\right)}<\frac{\varepsilon}{J'}.$$
 By the lemma \ref{LC}, we estimate
 \begin{align*}
\|(2)\|_{L_t^2 L_x^{\frac{2N}{N+2}}}\lesssim
\|e^{it\Delta}w_n^{J}\|_{S\left([0,\infty)\right)}^{\alpha-b}\|\nabla e^{it\Delta}w_n^{J}\|_{W(\R)}^{1+b}
\end{align*}
 and
 $$
 \|(3)\|_{L_t^2 L_x^{\frac{2N}{N+2}}}\lesssim \|e^{it\Delta}w_n^{J}\|_{S\left([0,\infty)\right)}^{\alpha-b}\|\nabla u_n^{J}\|_{W(\R)}\|\nabla e^{it\Delta}w_n^{J}\|_{W(\R)}^{b}.
 $$
Together with vanishing condition \eqref{vanishing} and $\|w_n^{J}\|_{\dot H^1}\lesssim1$, we have
$$\limsup_{n\to\infty}\big(\|(2)\|_{L_t^2 L_x^{\frac{2N}{N+2}}}+\|(3)\|_{L_t^2 L_x^{\frac{2N}{N+2}}}\big)\leq \varepsilon.$$

Now we turn to estimate $(4)$. For $0<b<1$, thus
$$
 \|(4)\|_{L_t^2 L_x^{\frac{2N}{N+2}}}\lesssim \|u_n^{J}\|_{S\left([0,\infty)\right)}^{\alpha-1}\|\nabla u_n^{J}\|_{W(\R)}\|\nabla e^{it\Delta}w_n^{J}\|_{W(\R)}^{b}\|e^{it\Delta}w_n^{J}\|_{S\left([0,\infty)\right)}^{1-b}.
 $$
Similarly, we have
$$\limsup_{n\to\infty}\|(4)\|_{L_t^2 L_x^{\frac{2N}{N+2}}}\leq \varepsilon.$$
For the setting $b\geq1$, we only estimate a single term of the form, for $j\in\{1,...,J\}$,
$$\||x|^{-b} |v_n^{j}|^{\alpha-1}\cdot e^{it\Delta}w_n^{J} \cdot \nabla v_n^{j}\|_{L_t^2L_x^{\frac{2N}{N+2}}}.$$
By density, we may assume $v^{j}\in C_c^\infty(\R\times\R^N\setminus\{0\})$. Applying H\"older's inequality, the problem further reduces to showing that
\begin{align}\label{YU}
  \lim_{J\to J^*}\limsup_{n\to\infty}\|\nabla \tilde{w}_n^{J}\|_{L_{t,x}^2(K)}=0\ \text{for any compact $K\subset \R\times \R^N$},
\end{align}
where $\tilde{w}_n^{J}:=(\lambda_n^j)^{\frac{N-2}2}e^{i((\lambda_n^j)^2t-t_n^j)\Delta}w_n^{J}(\lambda_n^jx+x_n^j).$
This finally follows from an interpolation argument using a local smoothing estimate and vanishing \eqref{vanishing} (See \cite[Lemma 2.12]{KV} for details). The term $(5)$ can be estimated as above, then we complete the proof.

\end{proof}

Returning to the proof of Proposition \ref{PS} and rearranging the indices, we may
assume that there exists $1 \leq J_1 < J_0$ such that
$$
\limsup_{n\rightarrow \infty} \|v^j_n\|_{S([0,T_n^j))} = \infty\;\; \textnormal{for}\;\; 1 \leq j \leq J_1\;\textnormal{and}\;\; \limsup_{n\rightarrow \infty}\|v^j_n\|_{S([0,\infty))} <\infty  \;\textnormal{for}\; j > J_1.
$$

For each $m, n \geq 1$ let us define an integer $k_{n,m}=k(n,m) \in \{1, . . . , J_1\}$ and an interval $I_n^m=[0,s]$ by
$$
\sup_{1\leq j\leq J_1}\|v_n^j\|_{S(I_n^m)} = \|v_n^{k_{n,m}}\|_{S(I_n^m)}=m.
$$
Applying the pigeonhole principle, there must exist a $j_1$ such that $k_{n,m}= j_1$ for infinitely many $n$. By reordering the indices, we can assume that $j_1 = 1$. Thus, 
$$
\limsup_{n,m\rightarrow \infty} \|v_n^1\|_{S(I_n^m)}=\infty,
$$
which implies (using \eqref{PSC2}) 
\begin{equation}\label{PSC3}
\limsup_{n,m \rightarrow \infty} \sup_{t\in I_n^m} \|\nabla v^1_n(t)\|^2_{L^2}\geq E_c.
\end{equation}

\ Notice that since all $v^j_n$ have finite scattering size on $I_n^m$ for each $m \geq 1$, we can obtain the same approximation result by employing a similar proof as before. More precisely,
$$
\lim_{J\rightarrow \infty} \limsup_{n\rightarrow \infty}\|u^J_n- u_n\|_{L^\infty_t \dot{H}^1_x(I_n^m )}=0.
$$
for each $m \geq 1$.

\ To conclude the proof of the proposition, we will utilize the following lemma, which will be proven subsequently.
\begin{lemma}{(\bf Kinetic energy decoupling for $u^J_n$)}\label{KED} For all $J\geq 1$ and $m\geq 1$,
$$
\limsup_{n\to\infty} \sup_{t\in I^m_n}\;\bigl| \|\nabla u_n^J(t)\|_{L^2}^2 - \sum_{j=1}^J \|\nabla v^j(t) \|_{L^2}^2 - \|\nabla w_n^J\|_{L^2}^2\bigr|=0.
$$
\end{lemma}

\ Applying the previous lemma, \eqref{Hyp1} and \eqref{PSC3} one has
$$
E_c\geq \limsup_{n\to\infty} \sup_{t\in I^m_n} \|\nabla u_n^J(t)\|_{L^2}^2=\lim_{J \rightarrow \infty}\limsup_{n\rightarrow \infty} \{ \|\nabla w^J_n\|^2_{L^2}+\sup_{t\in I_n^m}\sum_{j=1}^J \|\nabla v_n^j(t)\|^2_{L^2} 
\}.
$$
Therefore, the relation \eqref{PSC3} implies that $J_1 = 1$, $v^j_n \equiv 0$ for all $j \geq 2$, and $w_n$ converges to zero strongly in $\dot{H}^1$. Thus,
$$
u_n(0)=g_n^1[e^{it_n^1\Delta}\phi^1]+w^1_n.
$$
Finally, we show that the sequence $u_n(0)$ is precompact in $\dot{H}^1$. To this end, we must show that space-time translation parameters obey $(t^1_n,x^1_n) \equiv (0,0)$. That is, $u_n(0)-g^1_n(\phi^1)\rightarrow 0$ in $\dot{H}^1$.

\ We claim that $x^1_n \equiv 0$. Otherwise, if $|\frac{x^1_n}{\lambda_n^1}|$, then, by Proposition \ref{P:embed}, for all n sufficiently large, there exists a global solution $v^1_n$ to \eqref{INLS} satisfying
$$
v_n^1(0) = g_n^1[e^{it_n\Delta}\phi^1]\;\;\; and\;\;\; \|v_n^1\|<\infty.
$$
Thus, by the stability result (Proposition \ref{stability}), we show that $\|u_n\|_{S(\R)} < \infty$, which is a contradiction. We now suppose that $t^1_n\rightarrow \infty$. The Strichartz estimate and the monotone convergence theorem imply that
$$
\|e^{it\Delta}u_n(0)\|_{S\left([0,\infty)\right)}\lesssim \|e^{it\Delta}\phi^1\|_{S\left([t_n^1,\infty)\right)}+\|e^{it\Delta}w_n^1\|_{S\left([0,\infty)\right)}. 
$$
Therefore, $\|e^{it\Delta}u_n(0)\|_{S\left([0,\infty)\right)}$ is small for large $n$, using \eqref{vanishing}. Consequently, by the small data theory, we conclude that $\|u_n\|_{S\left([0,\infty)\right)}$ is bounded, which is a contradiction. Similarly, we obtain a contradiction when $t^1_n\rightarrow -\infty$.

We conclude the proof by proving the decoupling of kinetic energy. 

\begin{proof}{Lemma \ref{KED}} The proof is similar to that of Lemma 3.2 in \cite{KV}. So, we only give the main steps. Indeed, for any $t_n \in I_n^m$, a direct computation gives
$$
 \|\nabla u_n^J(t)\|_{L^2}^2 - \sum_{j=1}^J \|\nabla v^j(t) \|_{L^2}^2 - \|\nabla w_n^J\|_{L^2}^2=I_n+II_n,
$$    
where
$$
I_n=\sum_{j\neq k} \langle v_n^j(t_n^j),v_n^k(t_n^k)\rangle_{\Dot{H}^1}\qquad \textnormal{and}\qquad II_n=\sum_{j=1}^J 2\Re \langle e^{it\Delta}w_n^J, v_n^j(t_n^j)\rangle_{\Dot{H}^1}.
$$
It suffices for us to show
$$
\lim_{n\rightarrow \infty} I_n=0\;\;\textnormal{for}\;j\neq k\quad\textnormal{and}\quad \lim_{n\rightarrow \infty} II_n=0.
$$
To establish $I_n$, we employ the following result: for all $J \geq 1$ and $1 \leq j \leq J$, the sequence $e^{it^j_n\Delta}[(g^j_n)^{-1}w^J_n]$ converges weakly to zero in $\dot{H}^1$ as $n \rightarrow \infty$. The proof of $II_n$ follows a similar structure to that of $I_n$, integrating the orthogonality of parameters \eqref{orthogonality}.
\end{proof}

This completes the proof of Proposition \ref{PS}.
\end{proof}

\begin{proposition}\label{MS}
Assume Theorem \ref{T} fails, then there exist a critical value $0 < K_c < \|W\|^2_{\dot{H}}$ and a forward maximal life-span solution $u_c: [0, T_{\max}) \times \R^N
\rightarrow \C$ to \eqref{INLS} with 
$$
\sup_{t\in [0, T_{\max})}\|u_c(t)\|^2_{\dot{H}^1}=E_c<\|W\|^2_{\dot{H}^1}\quad \textnormal{and}\quad \|u_c\|_{S\left([0, T_{\max})\right)}=\infty.
$$
Moreover, there exists a frequency scale function $\lambda : [0, T_{\max}) \rightarrow (0, \infty)$ such that $$K=\{ \lambda(t)^{-\frac{(N-2)}{2} } u_c(t, \lambda(t)^{-1}x) : t \in [0, T_{\max})\}$$
is precompact in $\dot{H}^1$. An analogous result holds backward in time.

\end{proposition}
\begin{proof} The proof primarily relies on the Palais-Smale condition (more details, see \cite{KM}, \cite{KV}). We give the main steps. Assume Theorem \ref{T} fails. From the definition of $K_c$, there exists a sequence of solutions $u_n :I_n \times \R^N \rightarrow \C$ to \eqref{INLS} such that
\begin{equation*}
\lim_{n\rightarrow\infty}sup_{t\in I_n}\|u_n(t)\|^2_{\dot{H}^1}= Kc < \|W\|^2_{\dot{H}^1}\,,\qquad \lim_{n\rightarrow \infty} \|u_n\|_{S(t\geq t_n)} = \lim_{n\rightarrow \infty} \|u_n\|_{S(t\leq t_n)} = \infty.
\end{equation*}
By choosing $t_n \in I_n$ such that $\| u_n(t) \|_{S(t \geq t_n)} = \| u_n(t) \|_{S(t \leq t_n)}$, and exploiting time translation invariance, we set $t_n \equiv 0$. According to Proposition \ref{PS}, by passing to a subsequence, there exists $\{\lambda_n\} \subset \mathbb{R}^+$ such that $\lambda^{-\frac{N-2}{2}}u_n(0,\frac{x}{\lambda_n}) \rightarrow u_0 \text{ in } H^1.$

Let $u_c : I_{\max} \times \mathbb{R}^d \rightarrow \mathbb{C}$ with $u_c(0) = u_0$ be the maximal life-span solution to \eqref{INLS}. The stability result implies that for any compact interval $I \subset I_{\max}$,
\[
\begin{cases}
\lim_{n \rightarrow \infty}\| u_n - u\|_{L^\infty_t\dot{H}^1_x( I\times \R^N) } = 0 \\
\| u(t) \|_{S\left([0,\infty)\right)} = \| u(t) \|_{S\left((-\infty,0]\right)} = \infty \\
\sup_{t \in I} \| u(t) \|_{\dot{H}^1} = K_c < \| W\|^2_{\dot{H}^1}.
\end{cases}
\]
Moreover, using Proposition \ref{PS} again, we obtain that $K$ is precompact in $\dot{H}^1$.

\end{proof}

We now have all the tools to show the main theorem.

\subsection{Proof of Theorem \ref{T}} We proceed by contradiction. Suppose that Theorem \ref{T} fails. Then, by Proposition \ref{MS}, there exists $u_c : [0, T_{\max}) \times \mathbb{R}^N \rightarrow \mathbb{C}$, which is a minimal blow-up solution with a compactness property. We aim to show the non-existence of the minimal non-scattering solution $u_c$, as stated in Proposition \ref{MS}. To this end, we will consider two cases: $T_{\max} < \infty$ and $T_{\max} = \infty$. 

\ 

{\bf Claim $1$:} There are no solutions to \eqref{INLS} of the form given in Proposition \ref{MS} with $T_{\max} < \infty$.

\ 

Indeed, we use the following reduced Duhamel formula, which is a consequence of the compactness properties of $u_c$ (see \cite[Proposition~5.23]{KVClay} for details).
\begin{lemma}[Reduced Duhamel formula]\label{P:RD} For $t\in[0,T_{\max})$, the following holds as a weak limit in $\dot H^1$:
\[
u_c(t) = i\lim_{T\to T_{\max}} \int_t^T e^{i(t-s)\Delta}[|x|^{-b}|u_c|^{\alpha} u_c(s)]\,ds.
\]
\end{lemma}

We suppose that $T_{\max}<\infty$.  Then, using Lemma~\ref{P:RD}, Strichartz estimates, H\"older's inequality, Bernstein's inequality, and Hardy's inequality, we find that for any $t\in[0,T_{\max})$ and any $M>0$,
\begin{align*}
\|P_M u_c(t)\|_{L_x^2} & \lesssim \|P_M(|x|^{-b}|u_c|^{\alpha} u_c)\|_{L_t^1 L_x^2([t,T_{\max})\times\R^N)} \\
& \lesssim M^{N(\frac{N+2}{2N}-\frac{1}{2})}(T_{\max}-t)\| |x|^{-b}|u_c|^{\alpha}u_c\|_{L_t^\infty L_x^{\frac{2N}{N+2}}}\\
& \lesssim M(T_{\max}-t)\| u_c\|_{L_t^\infty L_x^{\frac{2N}{N-2}}}^{\alpha+1-b}\| \nabla u_c\|_{L_t^\infty L_x^2}^{b}.
\end{align*}
Thus, using Bernstein's inequality for the high frequencies, we deduce\footnote{Note that, $u=P_Mu+(1-P_M)u$.}
\[
\|u_c(t)\|_{L_x^2} \lesssim M(T_{\max}-t) + M^{-1}.
\]
for any $t\in[0,T_{\max})$ and $M>0$.  Using conservation of mass, we deduce $\|u_c\|_{L^2}\equiv 0$ and hence $u\equiv 0$. As $u$ is not identically zero, we conclude that the finite-time blowup scenario is impossible.

\ 

{\bf Claim $2$:} There are no solutions to \eqref{INLS} of the form given in Proposition \ref{MS} with $T_{\max} = \infty$.

\

Indeed, for a smooth weight $a$, we define
\begin{align}\label{M-Identity}
M_a(t) = 2\Im\int \bar u u_j a_j\,dx,
\end{align}
where we use subscripts to denote partial derivatives and sum repeated indices.  Using a computation using \eqref{INLS} and integration by parts, we then have
\begin{align}\label{virial}
\frac{dM_a}{dt} = &\int 4\Re a_{jk}\bar u_j  u_k - |u|^2 a_{jjkk} - \Big(2-\frac{4}{\alpha+2}\Big)|x|^{-b}|u|^{\alpha+2} a_{jj} \\\nonumber
&\ \ \ \ \ \ \ - \frac{4b}{\alpha+2}|x|^{-b-2}|u|^{\alpha+2} x_j a_j\,dx.
\end{align}
The standard virial identity corresponds to the choice $a(x)=|x|^2$. However, in this case, we cannot guarantee the finiteness of $M_a(t)$, as we are working with merely $\dot H^1$ data. Thus it is essential to localize the weight in space. The success of this relies on the compactness of the solution $u_c(t)$. In particular, we can establish the following.
\begin{lemma}[Tightness]\label{L:tight} Let $\eps>0$.  Then there exists $R=R(\eps)$ sufficiently large so that
\begin{equation}
\sup_{t\in[0,\infty)} \int_{|x|>R} \bigl\{|\nabla u_c(t,x)|^2 + |x|^{-b}|u_c(t,x)|^{\alpha+2}+ |x|^{-2}|u_c(t,x)|^2\bigr\}\,dx<\eps.
\end{equation}
\end{lemma}

\begin{proof} The proof is standard, so we omit the details.
\end{proof}

Returning to the proof of Claim 2. Fix $\eps>0$ and choose $R=R(\eps)$ as in the previous lemma. Choosing our weight $a$ such that
\[
a(x) = \begin{cases} |x|^2 &\text{for } |x|\leq R \\ CR^2 & \text{for } |x|>2R,\end{cases}
\]
for some $C>1$.  In the intermediate region, we can impose
\[
|\partial^{\alpha} a| \lesssim R^{2-|\alpha|} \qtq{for} R<|x|\leq 2R
\]
for any multiindex $\alpha$. H\"older's inequality, and Lemma~\ref{energy-trapping} implies that (for some $\delta>0$)
\[
\sup_{t\in[0,\infty)} |M_a(t)| \lesssim R^2 \|u_c\|_{L_t^\infty \dot H_x^1}^2\lesssim R^2 E(u_c).
\]
Furthermore, by \eqref{virial}, one has
\begin{align}
\frac{dM_a}{dt} =  &8 \int_{\R^N} |\nabla u_c|^2 - |x|^{-b}|u_c|^{\alpha+2}\,dx \label{virial-main} \\
& + \mathcal{O}\biggl\{ \int_{|x|>R} |\nabla u_c|^2 + |x|^{-b}|u_c|^{\alpha} + |x|^{-2}|u_c|^2 \,dx\biggr\} \label{virial-error}
\end{align}
Lemma~\ref{energy-trapping} and Lemma~\ref{L:coercive} implies that
\[
\eqref{virial-main} \geq \delta\int |\nabla u_c|^2\,dx \gtrsim \delta E(u_c),
\]
where $t\in[0,\infty)$ . On the other hand, Lemma~\ref{L:tight} implies $|\eqref{virial-error}|<\eps$. Hence, by applying the fundamental theorem of calculus and integrating throughout the form $[0,T]$, it follows that
\[
E(u_c) \lesssim \tfrac{R^2}{\delta T}E(u_c) + \eps.
\]
Choosing $T$ sufficiently large, we obtain $E(u_c)\equiv 0$ and so $u_c\equiv 0$. This is a contradiction.






\

\section{Blow-up}\label{blowup}

This section shows the blow-up result via localized virial identity (Theorem \ref{Blow-up}. To this end, we first show the local well-posedness in $H^1(\R^N)$. The proof closely follows Proposition \ref{GWPCH1}, with the key difference being the substitution of the metric $d(u,v)=\|u-v\|_{S^1(I)}$ with $d(u,v)=\| u-v\|_{W(I)}$. It is worth mentioning that the metric is determined by the function itself, not by its derivative. As a result, we avoid the dimension restriction typically imposed by $\alpha \geq 1$, allowing the result to hold for all dimensions $N \geq 3$. In the sequel, we show a technical lemma and finally, we prove Theorem \ref{Blow-up}.

\begin{lemma}\label{LLWPH1}
Assume $N\geq3$, $\alpha=\frac{4-2b}{N-2}$ and $0<b\leq \frac{4}{N}$. For any $u_0 \in H^1(\mathbb{R}^N)$, there  exists $T = T (u_0)$ and a unique solution $u$ with initial data $u_0$ satisfying
$$
u \in L^q_{loc}\left((-T, T \right), H^{1,r}),\;\;\; \forall\; (q, r) \;\;S\textnormal{-admissible}.
$$
  
\end{lemma}

\begin{proof}
Define
$$
S_{T,\rho}=\{u\in C([0,T];H^1(\mathbb{R}^N))\;:\;\,\|u\|_{S^1([0,T])}\leq \rho\,,\,\,\|u\|_{W([0,T])}\leq M \},
$$
where $\|\cdot\|_{S^1(I)}$ as in \eqref{normS1}. We show that the operator $G$ defined in \eqref{OPERATOR} is a contraction on $S_{T,\rho}$ equipped with the metric
$$
d(u,v):=\| u-v\|_{W([0,T])}.
$$

\ The Sobolev embedding and the Strichartz estimates \eqref{SE1}-\eqref{SE2} yield
\begin{eqnarray*}
\|G(u)\|_{S([0,T])}
\leq  \|e^{it\Delta}u_0\|_{S([0,T])}+c\left\|\nabla F(x,u)\right\|_{L_I^{2}L_x^{\frac{2N}{N+2}}},
\end{eqnarray*}
\begin{equation*}
\|\nabla G(u)\|_{W([0,T])}\leq \|\nabla e^{it\Delta}u_0\|_{W([0,T])}+ c\|\nabla F(x,u)\|_{L^2_IL_x^{\frac{2N}{N+2}}},
\end{equation*}
and
\begin{equation*}
\| G(u)\|_{W([0,T])}\leq \|\nabla e^{it\Delta}u_0\|_{W([0,T])}+ c\| F(x,u)\|_{L^2_IL_x^{\frac{2N}{N+2}}}.
\end{equation*}
Thus, applying Lemma \ref{LC} together with the Hardy inequality one has
\begin{equation*}
\begin{split}
\|G(u)\|_{S^1([0,T])}&\leq \|e^{it\Delta}u_0\|_{S^1([0,T])} +c \|u\|^{\alpha-b}_{S([0,T])}\|\nabla u\|^{b+1}_{W([0,T])}\\
\|G(u)\|_{W([0,T])}&\leq  \|e^{it\Delta}u_0\|_{W([0,T])}+c \|u\|^{\alpha-b}_{S([0,T])}\|\nabla u\|^{b}_{W([0,T])}\|u\|_{W([0,T])}
\end{split}
\end{equation*}
If $u\in S_{T,\rho}$ we obtain
\begin{align*}
\|G(u)\|_{W([0,T])}\leq& 
c\|u_0\|_{H^1}+c\rho^{\alpha}M.
\end{align*}
Let $M=2c\|u_0\|_{H^1}$ and choose $\rho $ such that 
$c\rho^{\alpha}< \frac{1}2.$ Moreover, choosing $T$ small enough such that 
$\| e^{it\Delta}u_0\|_{S^1([0,T])}< \frac{\rho}{2}$, we get 
$$\| G(u)\|_{S^1([0,T]))}\leq \rho \qquad\textnormal{and}\qquad \| G(u)\|_{W([0,T]))}\leq M.$$
That is, $G(u)\in S_{T,\rho}$.

Repeating the above computations we get (if $u, v \in S_{T,\rho}$)
\begin{eqnarray*}
d(G(u),G(v))&\leq& c \|F(x,u)- F(x,v)\|_{L^2_IL^{\frac{2N}{N+2}}_x}\\
&\leq& c\;\left( \|\nabla u\|^b_{W(I)} \|u\|^{\alpha-b}_{S(I)}+\|\nabla v\|^b_{W(I)} \|v\|^{\alpha-b}_{S(I)}\right) \|u-v\|_{W(I)}\\
&\leq& 2c\rho^\alpha d(u,v).
\end{eqnarray*}
So, by the contraction mapping principle, $G$ has a unique fixed point $u\in S_{T,\rho}$.
\end{proof}

We now show the following lemma.

\begin{lemma}\label{Blow}
  Let $u(t):I_{\max}\times \R^N\to\mathbb{C}$ be a solution to \eqref{INLS} obeying
  $$E[u_0]<(1-\delta_0)E[W]\qquad  \textnormal{and} \qquad \|\nabla u_0\|_{L^2}\geq\|\nabla W\|_{L^2},$$
  there exists $\delta$ such that
  \begin{align*}
    \ \|u(t)\|_{\dot H^1}\geq (1+\delta)\|W\|_{\dot H^1}
  \end{align*}
for all $t\in I_{\max}.$
\end{lemma}
\begin{proof}
By Lemma \ref{energy-trapping}, we have
\[
F(y)=\frac{y}2-\tfrac{c^{-\frac{\alpha}2}}{\alpha+2}y^{\frac{\alpha+2}{2}}, \qtq{while} y\geq c
\]
with $y=\|u(t)\|_{\dot H^1}^2$. Using the fact $F(0)<(1-\delta_0)\tfrac{\alpha c}{2(\alpha+2)}$, continuity of the flow in $\dot H^1$, and
conservation of energy. Thus, one can get  $y>c+\delta_1$ for some $\delta_1$.
Therefore we get the desired result.
\end{proof}

Finally, we show the blow-up result for \eqref{INLS} with non-radial initial data. 
\begin{proof}[\bf Proof of Theorem \ref{Blow-up}]
Let $\varphi(x)\in C_0^{\infty}(\R^N)$ be a radial function, denote by
\begin{align*}
  \varphi(x)=
  \begin{cases}
    &\frac{|x|^2}{2},  \text{if}\ |x|\leq 1;\\
    &0, \quad \text{if}\ |x|\geq 10,
  \end{cases}
  \ \text{and}\ \varphi''(x)\leq1.
\end{align*}
Then, let $R\geq1$ and $a(x)=R^2\varphi(x/R)$ satisfies
$$a''(x)\leq1,\ a'(x)\leq R\  \text{and}\ \Delta a=N$$
and, for $|x|\leq R$, we have
$$\Delta a(x)=N\  \text{and}\ \nabla a(x)=x.$$

Set $V_R(t)=\int_{\R^N}a(x)|u|^{2}dx$, then
$$\frac{d}{dt}V_R(t)=M_a(t).$$
From \eqref{M-Identity}, one can obtain that
\begin{align*}
  \frac{d^2}{dt^2}V_R(t)=&4\int_{\R^N} \Re a_{jk}\bar u_j  u_k - |u|^2 a_{jjkk}\\
  &- \Big(2-\frac{4}{\alpha+2}\Big)|x|^{-b}|u|^{\alpha+2} a_{jj}- \frac{4b}{\alpha+2}|x|^{-b-2}|u|^{\alpha+2} x_j a_j\,dx\\
 =&4\int_{\R^N}|\nabla u|^2\frac{a'(x)}{|x|}dx+4\int_{\R^N}|x\cdot\nabla u|^2\Big(\frac{a''(x)}{|x|^2}-\frac{a'(x)}{|x|^3}\Big)dx-\int_{\R^N}|u|^2 a_{jjkk}dx\\
 &-\int_{\R^N} \Big(2-\frac{4}{\alpha+2}\Big)|x|^{-b}|u|^{\alpha+2} a_{jj}- \frac{4b}{\alpha+2}|x|^{-b-2}|u|^{\alpha+2} x_j a_j\,dx
\end{align*}
Since $a''(x)\leq1$ and $a'(x)=|x|~(|x|\leq R)$, we then have
\begin{align*}
\frac{d^2}{dt^2}V_R(t)\leq&4\int_{\R^N}|\nabla u|^2dx-\int_{\R^N}|u|^2 a_{jjkk}dx\\
 &-\int_{\R^N} \Big(2-\frac{4}{\alpha+2}\Big)|x|^{-b}|u|^{\alpha+2} a_{jj}- \frac{4b}{\alpha+2}|x|^{-b-2}|u|^{\alpha+2} x_j a_j\,dx\\
 \leq&4\int_{\R^N}|\nabla u|^2dx-4\int_{\R^N}|x|^{-b}|u|^{\alpha+2}dx\\
 &+\frac{C}{R^2}\int_{R\leq |x|\leq 10R}|u|^2dx+CR^{-b}\int_{|x|\geq R}|u|^{\alpha+2}dx.
\end{align*}

To control the last term we use the Gagliardo-Nirenberg inequality:
$$\|f\|_{L_x^{\alpha+2}}^{\alpha+2}\leq C\|f\|_{L_x^2}^{\frac{N+2-(N-2)(\alpha+1)}2}\|\nabla f\|_{L_x^2}^{\frac{N\alpha}2}.$$
By the mass conservation, we can get
\begin{align}\label{V_R}
 \frac{d^2V_R}{dt^2}(t)\leq&4\Big(\int_{\R^N}(|\nabla u|^2-|x|^{-b}|u|^{\alpha+2})dx\Big)+CR^{-b}\|\nabla u\|_{L_x^2}^{\frac{N\alpha}2}+CR^{-2},
\end{align}
where $C$ depends on $\|u_0\|_{L^2}$ and we use the fact $0<b<2$.

\ 

\textbf{Case I: $b\geq\frac{4}N$.} From inequality \eqref{V_R}, then one can obtain that
\begin{align*}
  \frac{d^2}{dt^2}V_R(t)\leq&4\int_{\R^N}|\nabla u|^2dx-4\int_{\R^N}|x|^{-b}|u|^{\alpha+2}dx+R^{-b}\|\nabla u\|_{L_x^2}^2+CR^{-b}\\
  \leq&4(\alpha+2)E(u(t))-(2\alpha-R^{-b})\|\nabla u\|_{L_x^2}^2+CR^{-b}
\end{align*}
where $C$ depends on $\|u_0\|_{L_x^2}$. By Lemma \ref{Blow} we deduce that
\begin{align*}
 \frac{d^2}{dt^2}V_R(t)\leq& 2\alpha(1-\delta_0)\|\nabla W\|_{L_x^2}^2-(1+\delta)(2\alpha-R^{-b})\|\nabla W\|_{L_x^2}^2+CR^{-\min\{b,2\}}\\
 \leq&-2\alpha(\delta_0+\delta)\|\nabla W\|_{L_x^2}^2+(1+\delta)R^{-b}\|\nabla W\|_{L_x^2}^2+CR^{-\min\{b,2\}}.
\end{align*}
Choosing $R$ large such that
$$(1+\delta)R^{-b}\|\nabla W\|_{L_x^2}^2+CR^{-\min\{b,2\}}\leq\alpha(\delta_0+\delta)\|\nabla W\|_{L_x^2}^2, $$
then we have
$$\frac{d^2}{dt^2}V_R(t)\leq-\alpha(\delta_0+\delta)\|\nabla W\|_{L_x^2}^2.$$
Thus, by integrating, it follows
\begin{align*}
  \int_{\R^N}a(x)|u(t)|^2dx\leq -Ct^2+2tM_a(t)+\int_{\R^N}a(x)|u(t)|^2dx
\end{align*}
for $C=\alpha(\delta_0+\delta)\|\nabla W\|_{L_x^2}^2$. The inequality give us that the maximal lifespan is bounded. Therefore we get the desired result.

\

\textbf{Case II: $0<b<\frac{4}N$}. Suppose $I_{\max}=\R$. We will prove $u(t)$ blows-up in infinite time. Fixing $T<\infty$, let $R(t)$ denote by
$$R(T)=\sup_{[0,T]}(\delta^{-1}\|\nabla u(t)\|_{L_x^2})^{\frac{N\alpha-4}{2}}.$$
Thus, we have
$$R(T)^{-b}\|\nabla u(t)\|_{L_x^2}^{\frac{N\alpha}2}\leq \delta \|\nabla u(t)\|_{L_x^2}^2$$
and (using the fact $\|u(t)\|_{L_x^2}^2\geq C$)
$$R(T)^{-b}\leq C\delta \|\nabla u(t)\|_{L_x^2}^2.$$
Together with \eqref{V_R}, Lemma \ref{Blow} and choosing $\delta$ small enough, we can get
$$V''_R(t)\lesssim -\delta_0\|\nabla u(t)\|_{L_x^2}^2.$$

A direct computation shows that
$$V_R(t)=V_R(0)+V'_R(0)T+\int_{0}^T\int_{0}^s V^{''}_R(t)dtds.$$
Using the definition of $V_R(t)$, thus we can get
$$V_R(0)\lesssim_{\|u_0\|_{L_x^2}} R(T)^2$$
and
$$|V^{'}_R(0)|\lesssim_{\|\nabla u_0\|_{L_x^2}}R(T).$$
Collecting the above estimates, we have
\begin{align*}
  \int_{0}^T\int_{0}^s \|\nabla u(t)\|_{L_x^2}^2dtds\lesssim R(T)^2+R(T)\cdot T.
\end{align*}
Using the fact $\|\nabla u(t)\|_{L_x^2}\geq C$ and Young's inequality, thus
$$T^2\lesssim R(T)^2+\frac{R(T)^2}2+\frac{T^2}2.$$
Therefore, we can obtain that $R(T)\geq CT$. Hence
\begin{align*}
  \sup_{t\in [0,T]}\|\nabla u(t)\|_{L_x^2}\geq C\delta T^{\frac{2}{N\alpha-4}}.
\end{align*}
We complete the proof.
\end{proof}

\section*{Data availability statement}

No new data were created or analyzed during this study. Data sharing is not applicable to this article. 
 

\end{document}